\documentclass[12pt]{amsart}
\usepackage{amssymb}
\newtheorem{thm}{Theorem}
\newtheorem{lem}{Lemma}[section]
\newtheorem{prop}{Proposition}[section]
\newtheorem{cor}{Corollary}
\newtheorem{conj}{Conjecture}
\newtheorem{definition}{Definition}
\newtheorem{prob}{Problem}
\usepackage{url}
\usepackage{ulem}
\usepackage[usenames]{color}
\usepackage[utf8]{inputenc}

\numberwithin{equation}{section}

\newcommand{\N}{\mathbb{N}}
\newcommand{\Z}{\mathbb{Z}}

\DeclareMathOperator{\lcm}{lcm}

\begin{document}
\title[special type of unit equations in two variables II]
{Number of solutions to a special type \\of unit equations in two unknowns, II}

\author{Takafumi Miyazaki}
\address{Takafumi Miyazaki
\hfill\break\indent Gunma University, Division of Pure and Applied Science,
\hfill\break\indent Faculty of Science and Technology
\hfill\break\indent Tenjin-cho 1-5-1, Kiryu 376-8515.
\hfill\break\indent Japan}
\email{tmiyazaki@gunma-u.ac.jp}

\author{Istv\'an Pink}
\address{Istv\'an Pink
\hfill\break\indent University of Debrecen, Institute of Mathematics
\hfill\break\indent H-4002 Debrecen, P.O. Box 400.
\hfill\break\indent Hungary}
\email{pinki@science.unideb.hu}

\thanks{The first author is supported by JSPS KAKENHI (No.20K03553).
The second author was supported by the NKFIH grants ANN130909 and K128088.}


\date{\today}

\maketitle

\markleft{T. Miyazaki \& I. Pink}
\markright{
Special type of unit equations in two unknowns, II}


\begin{abstract}
This paper contributes to the conjecture of R.~Scott and R.~Styer which asserts that for any fixed relatively prime positive integers $a,b$ and $c$ all greater than 1 there is at most one solution to the equation $a^x+b^y=c^z$ in positive integers $x,y$ and $z$, except for specific cases.
The fundamental result proves the conjecture under some congruence condition modulo $c$ on $a$ and $b$.
As applications the conjecture is confirmed to be true if $c$ takes some small values including the Fermat primes found so far, and in particular this provides an analytic proof of the celebrated theorem of Scott [R. Scott, On the equations $p^x-b^y=c$ and $a^x+b^y=c^z$, J. Number Theory 44(1993), no.2, 153-165] solving the conjecture for $c=2$ in a purely algebraic manner.
The method can be generalized for smaller modulus cases, and it turns out that the conjecture holds true for infinitely many specific values of $c$ not being perfect powers.
The main novelty is to apply a special type of the $p$-adic analogue to Baker's theory on linear forms in logarithms via a certain divisibility relation arising from the existence of two hypothetical solutions to the equation.
The other tools include Baker's theory in the complex case and its non-Archimedean analogue for number fields together with various elementary arguments through rational and quadratic numbers, and extensive computation.
\end{abstract}

\section{Introduction} \label{sec-intro}%

The object here is the best possible general estimate of the number of solutions to a special type of the unit equations in two unknowns over the rationals.
The contents of this paper are motivated by some works of Scott \cite{Sc} and Bennett \cite{Be_cjm_01}, and they complement our previous work in \cite{MiPi}, but are independent of it.

First of all, we shall give a brief history on purely exponential Diophantine equations, and start with a general type, given as follows:
\begin{equation} \label{general}
a_1 {b_{11}}^{x_{11}} \cdots {b_{1 l}}^{x_{1 l}}
+a_2 {b_{21}}^{x_{21}} \cdots {b_{2 l}}^{x_{2 l}}
+ \cdots +a_k {b_{k1}}^{x_{k1}} \cdots {b_{k l}}^{x_{k l}}
=0
\end{equation}
with $k \ge 3$ and $l \ge 1$, where each letter using $a,b$ and $x$ denotes a fixed nonzero integer, a fixed integer greater than 1 and an unknown non-negative integer, respectively.
The above equation has a long and rich history, and includes several cases which have been actively studied to date (cf.~\cite[Ch.1]{ShTi}, \cite[D10]{Gu}, \cite{BerHa}, \cite[Ch.s 4 to 6]{EvGy}).
Since the number of prime divisors of each term in the left-hand side is finite, apparently equation \eqref{general} is a special case of unit equations which are a very important object in Diophantine number theory and appear in a number of topics concerning usual polynomial Diophantine equations as well (cf.~\cite{EvGy}).
Schmidt Subspace Theorem is applied for the unit equations to conclude that
equation \eqref{general} has at most finitely many solutions $x_{i j}\,(1 \le i \le k, 1 \le j \le l)$ for which the left-hand side has no vanishing subsum, and the finiteness of solutions for general unit equations has been extensively investigated in the literature (cf.~\cite[Ch.6]{EvGy}).
Although it is in general not easy to find all solutions to even very special cases of \eqref{general}, only for the case where the number of the terms in the equation is least, that is, $k=3$, Baker's theory on linear forms in logarithms is applied in general to give an explicit upper bound, being effectively computable by the equation's parameters, for each unknown exponent.
This fact plays a fundamental role to explicitly resolve exponential Diophantine equations including \eqref{general} with $k=3$ in many of the existing works, which is also the case for this paper.

From now on, we consider special cases of \eqref{general} with $k=3$, where the coefficients do not appear and the number of unknown exponents is very small.
They are closely related to the generalized Fermat conjecture (cf.~\cite{BeMihSi}) and Catalan's conjecture (Mih\u{a}ilescu's theorem \cite{Mi}) as well as to the well-known conjecture of Pillai which asserts that there are only finitely many pairs of distinct perfect powers with their difference fixed.
One of such examples is the following:
\begin{equation} \label{pillai}
a^x-b^y=c,
\end{equation}
where $a,b,c$ are fixed positive integers with $a>1$ and $b>1$, and $x,y$ are unknown positive integers.
Note that the unknown exponents here are allowed to equal 1.
After the pioneer works of Piilai \cite{Pi,Pi2} on the above equation, several researchers have attempted to obtain general estimate on the number of solutions for it (cf.~\cite[Sec.1]{Be_cjm_01}, \cite[Sec.1]{Be_JNT_03}).
A general and definitive result for this direction was finally obtained by Bennett \cite[Theorem 1.1]{Be_cjm_01} as follows:

\begin{prop} \label{atmost2pillai}
There are at most two solutions to equation $\eqref{pillai}.$
\end{prop}

The proof of this proposition is achieved by the combination of a special type of Baker's method on lower bounds for linear forms in two logarithms together with a gap principle arising from the existence of three hypothetical solutions.
It should be remarked that there are a number of examples which allow equation \eqref{pillai} to have two solutions (cf.~\eqref{excep-equs} below), and that the case in which the coprimality of $a$ and $b$ lacks is handled only by a simple (but skillful) argument.
In the same paper, as a further problem, Bennett posed a candidate for the complete list composed of the triples $(a,b,c)$ where there are two solutions to equation \eqref{pillai} (cf.~\cite[Conjecture 1.2]{Be_cjm_01}), and he gave a few partial results to support the validity of his question.

On the other hand, motivated by the celebrated theorem of Bennett (Proposition \ref{atmost2pillai}), we have attempted to generalize it to a 3-variable case, which concerns another particular case of equation \eqref{general} with $k=3$, given as follows:
\begin{equation} \label{abc}
a^x+b^y=c^z,
\end{equation}
where $a,b,c$ are fixed relatively prime positive integers greater than 1, and $x,y,z$ are unknown positive integers.
The above equation itself has a long history and has been actively studied by many researchers including pioneers Scott, Terai and Le etc. (see for example \cite{CiMi,Lu_aa_12,Mi_aa18} and the references therein).
The main result of our previous work \cite{MiPi} is as follows:

\begin{prop}\label{atmost2}
There are at most two solutions to equation $\eqref{abc},$ except when $(a,b,c)$ or $(b,a,c)$ equals $(5,3,2),$ where there are exactly three solutions.
\end{prop}

The proof of this proposition is achieved by the combination of Baker's method in both complex and $p$-adic cases together with an improved version of the gap principle established by Hu and Le \cite{HuLe,HuLe2,HuLe3} arising from the existence of three hypothetical solutions as well as a certain 2-adic argument relying upon the striking result of Scott and Styer \cite{ScSt}.
From the viewpoint of the generalized Fermat equation (cf.~\cite[Ch.14]{Co}), or of the classical popular problem to seek for all relations that the sum of two perfect powers being `relatively prime' equals another perfect power, Proposition \ref{atmost2} is regarded as a 3-variable version of Proposition \ref{atmost2pillai}.
Further it is definitive in the sense that there are (infinitely) many triples $(a,b,c)$ which allow equation \eqref{abc} to have two solutions.
Indeed, according to \cite[Sec.3]{ScSt}, we have
\begin{gather}
3^{}+5^{}=2^{3}, \ 3^{3}+5^{}=2^{5}, \ 3^{}+5^{3}=2^{7}; \nonumber\\
3^{}+13^{}=2^{4}, \ 3^{5}+13^{}=2^{8}; \nonumber\\
2^{2}+5^{}=3^{2}, \ 2^{}+5^{2}=3^{3}; \nonumber\\
2^{}+7^{}=3^{2}, \ 2^{5}+7^{2}=3^{4}; \nonumber\\
2^{3}+3^{}=11^{}, \ 2^{}+3^{2}=11^{}; \nonumber\\
3^{}+10^{}=13^{}, \ 3^{7}+10^{}=13^{3}; \nonumber\\
\label{excep-equs} 2^{5}+3^{}=35^{}, \ 2^{3}+3^{3}=35^{};\\
2^{}+89^{}=91^{}, \ 2^{13}+89^{}=91^{2}; \nonumber\\
2^{7}+5^{}=133^{}, \ 2^{3}+5^{3}=133^{};\nonumber\\
2^{8}+3^{}=259^{}, \ 2^{4}+3^{5}=259^{}; \nonumber\\
3^{7}+13^{}=2200^{}, \ 3^{}+13^{3}=2200^{}; \nonumber\\
2^{13}+91^{}=8283^{}, \ 2^{}+91^{2}=8283^{}; \nonumber\\
2^{}+(2^k-1)^{}={2^k+1}^{}, \ 2^{k+2}+(2^k-1)^{2}=(2^k+1)^{2},\nonumber
\end{gather}
where $k$ is any integer with $k \ge 2$.
On the other hand, similarly to the situation of equation \eqref{pillai}, something rather stronger than Proposition \ref{atmost2} seems to be true.
Based on a computer search, Scott and Styer \cite{ScSt} conjectured that \eqref{excep-equs} provides the complete list of all triples $(a,b,c)$ where there are at least two solutions to equation \eqref{abc}, as follows:

\begin{conj} \label{atmost1conj}
Assume that none of $a,b$ and $c$ is a perfect power.
There is at most one solution to equation $\eqref{abc},$ except when $(a,b,c)$ or $(b,a,c)$ belongs to the following set\,$:$
\begin{align}\label{excep-set}
\{ \,&(3,5,2),(3,13,2),(2,5,3),(2,7,3),\\
&(2,3,11),(3,10,13),(2,3,35),(2,89,91),(2,5,133),\nonumber\\
&(2,3,259),(3,13,2200),(2,91,8283),(2,2^k-1,2^k+1)\,\},\nonumber
\end{align}
where $k$ is any positive integer with $k=2$ and $k \ge 4.$
\end{conj}

This conjecture is regarded as a 3-variable version of Bennett's open question mentioned on equation \eqref{pillai}, and it seems to be an ultimate proposition through the studies on purely exponential Diophantine equations.
The aim of this paper is to establish several results on Conjecture \ref{atmost1conj}.
Note that it is already known from the literature that the solutions to equation \eqref{abc} corresponding to each triple in \eqref{excep-set} are described as in \eqref{excep-equs}.

Before stating our results, we introduce a simple notion, which was helpful in the proof of Proposition \ref{atmost2}, to extend the definition of multiplicative order on the reduced residue class groups.

\begin{definition} \label{extend}
Let $M$ be a positive integer.
For any integer $\mathcal A$ coprime to $M,$ the extended multiplicative order of $\mathcal A$ modulo $M$ is defined as the least positive integer $E$ such that ${\mathcal A}^E$ is congruent to $1$ or $-1$ modulo $M.$
\end{definition}

Our first result is the fundamental work of this paper.

\begin{thm} \label{th1}
Assume that the extended multiplicative orders of $a$ and $b$ modulo $c$ are relatively prime.
Then Conjecture $\ref{atmost1conj}$ is true, namely, there is at most one solution to equation $\eqref{abc},$ except when $(a,b,c)$ or $(b,a,c)$ equals one of $(3,5,2),(3,13,2),(2,5,3)$ and $(2,7,3).$
\end{thm}

We note that for the triples $(a,b,c)$ in \eqref{excep-set} not being handled by the above theorem, the values of $\gcd(\,e_{c}(a), e_{c}(b)\,)\,(>1)$ are given as follows:
\begin{align*}
&e_{ 11 }( 2 )= 5, & &e_{ 11 }( 3 )= 5 , & & \gcd (\,e_{ 11 }( 2 ),e_{ 11 }( 3 ) \,)= 5
,\\
&e_{ 13 }( 3 )= 3 , & &e_{ 13 }( 10 )= 3 , & & \gcd (\,e_{ 13 }( 3 ),e_{ 13 }( 10 ) \,)= 3
,\\
&e_{ 35 }( 2 )= 12 , & &e_{ 35 }( 3 )= 12 , & & \gcd (\,e_{ 35 }( 2 ),e_{ 35 }( 3 ) \,)= 12
,\\
&e_{ 91 }( 2 )= 12 , & &e_{ 91 }( 89 )= 12 , & & \gcd (\,e_{ 91 }( 2 ),e_{ 91 }( 89 ) \,)=
12 ,\\
&e_{ 133 }( 2 )= 18 , & &e_{ 133 }( 5 )= 18 , & & \gcd (\,e_{ 133 }( 2 ),e_{ 133 }( 5 )
\,)= 18 ,\\
&e_{ 259 }( 2 )= 36 , & &e_{ 259 }( 3 )= 9 , & & \gcd (\,e_{ 259 }( 2 ),e_{ 259 }( 3 ) \,)=
9 ,\\
&e_{ 2200 }( 3 )= 20 , & &e_{ 2200 }( 13 )= 20 , & & \gcd (\,e_{ 2200 }( 3 ),e_{ 2200 }( 13
) \,)= 20 ,\\
&e_{ 8283 }( 2 )= 25 , & & e_{ 8283 }( 91 )= 25 , && \gcd (\,e_{ 8283 }( 2 ),e_{ 8283 }( 91
) \,)= 25,\\
&e_{ 2^k+1 }( 2 )= k+1 , & & e_{ 2^k+1 }( 2^k-1 )= k+1 , && \gcd (\,e_{ 2^k+1 }( 2 ),e_{ 2^k+1 }( 2^k-1
) \,)= k+1.
\end{align*}

Theorem \ref{th1} has the following immediate corollary.

\begin{cor}\label{coro1}
Assume that at least one of $a$ and $b$ is congruent to $1$ or $-1$ modulo $c.$
Then there is at most one solution to equation $\eqref{abc},$ except when $(a,b,c)$ or $(b,a,c)$ equals one of $(3,5,2),(3,13,2),(2,5,3)$ and $(2,7,3).$
\end{cor}

Actually it will turn out that the above corollary is essentially equivalent to the first theorem (cf.~Section \ref{sec-reduce}).
The work for proving Corollary \ref{coro1} was motivated by attempting to obtain a 3-variable generalization of Bennett's result \cite[Theorem 1.6]{Be_cjm_01} which seems to be motivated for verifying his open question on equation \eqref{pillai} when $a$ takes fixed values.
It is worth noting that Corollary \ref{coro1} proves Conjecture \ref{atmost1conj} for $c=2,3$ and 6, and in particular this provides an analytic proof of the celebrated theorem of Scott \cite[Theorem 6;\,$p=2$]{Sc} solving Conjecture \ref{atmost1conj} for $c=2$ in a purely algebraic manner over imaginary quadratic fields, which is one of the most important novelties of this paper.

To state the next result, which is a non-explicit but effective generalization of Theorem \ref{th1}, we prepare some notation.
For a finite set $S$ of prime numbers, we denote by $\mathcal A[S]$ the $S$-{\it part} of a nonzero integer $\mathcal A$, namely,
\[
\mathcal A[S]=\prod_{p \in S} p^{\,\nu_p(\mathcal A)},
\]
where $\nu_p$ denotes the $p$-adic valuation.
For simplicity and convenience, we write $\mathcal A[\{p\}]=\mathcal A[p]$ and $\mathcal A[\emptyset]=1$, respectively.

\begin{thm}\label{th2}
Let $S$ be a $($possibly empty$)$ set of odd prime factors of $c.$
Define $M_S$ and $c_S$ as either
\begin{align*}
&M_S=\prod_{p \in S}\,p, \ \ c_S=\max\{c[S],c[2]\}; \ \ \ \text{or} \tag{I}\\
&M_S=4\prod_{p \in S}\,p, \ \ c_S=\frac{1}{2}\,c[ S \cup \{2\} ]. \tag{II}
\end{align*}
Assume that the extended multiplicative orders of $a$ and $b$ modulo $M_S$ are relatively prime and that $c_S>\sqrt{c}.$
Then there is at most one solution to equation $\eqref{abc},$ except when $(a,b,c)$ satisfies at least one of the following two restrictions\,{\rm :}
\begin{align*}
&\bullet \, \max\{a,b,c\}<\mathcal C_1;\\
&\bullet \, \max\{a,b\}<\exp\biggl(\frac{\mathcal C_2}{(\log c_S)/\log \sqrt{c}\,-1}\biggl), \ \ c_S<\exp(\mathcal C_2) \sqrt{c},
\end{align*}
where $\mathcal C_1$ and $\mathcal C_2$ are some positive absolute constants which are effectively computable.
\end{thm}

Theorem \ref{th1} is (almost) included in the above theorem for the case where $S$ is the set of odd prime factors of $c$.
Theorem \ref{th2} has a fruitful application to cases where $c$ takes fixed values.
The following two results are obtained by applying Theorem \ref{th2} for the case where $S$ is the intersection of the set of prime factors of $c$ and $\{3\}$, and for setting $c=p^n \cdot k$ with $p \in \{2,3\}$, $k$ a positive integer prime to $p$ and $n$ suitably large relative to $k$, respectively.

\begin{cor}\label{c2c3}
For any fixed $c$ satisfying $\max\{c[2],c[3]\}>\sqrt{c},$ there is at most one solution to equation $\eqref{abc},$ except for only finitely many pairs of $a$ and $b.$
\end{cor}

\begin{cor}
Conjecture $\ref{atmost1conj}$ is true for infinitely many values of $c$ which are not perfect powers.
\end{cor}

Note that Corollary \ref{c2c3} says for each of $c$ with $\max\{c[2], c[3]\}>\sqrt{c}$ that all pairs of $a$ and $b$ for which equation \eqref{abc} has more than one solution can be effectively found in a finite number of steps.
Also, we emphasize that from our method it is difficult to obtain a uniform version on $c$ of Corollary \ref{c2c3}, and further it seems to be hopeless to handle all exceptional triples appearing from such a version.

Our final result, together with Corollary \ref{coro1}, confirms Conjecture \ref{atmost1conj} if $c$ is one of the Fermat primes found so far.
This is regarded as a 3-variable generalization of \cite[Corollary 1.7]{Be_cjm_01}.

\begin{thm}\label{th3}
If $c \in \{5,17,257,65537\},$ then Conjecture $\ref{atmost1conj}$ is true, namely, there is at most one solution to equation $\eqref{abc},$ except when $(a,b)$ or $(b,a)$ equals $(2,c-2).$
\end{thm}

The organization of this paper is as follows.
In the next section we show some ideas to reduce each theorem to one of its weak forms.
In each of the proofs of our theorems, we examine solutions to the system of two equations arising from exceptional triples $(a,b,c)$ which allow equation \eqref{abc} to have at least two solutions.
Towards the proof of Theorem \ref{th1}, Sections \ref{sec-th1-pre} and \ref{sec-th1-bound} are respectively devoted to finding several congruence restrictions for solutions, and to finding upper bounds for them in several cases by applying Bugeaud's results in \cite{Bu} on simultaneous non-Archimedean valuations on the difference between two powers of algebraic numbers.
Here the most important idea is found through applying Baker's method in its non-Archimedean analogue to a certain divisibility relation among solutions which plays a crucial role in the proof of Proposition \ref{atmost2}.
These results together leave us a finite search for proving Theorem \ref{th1}, and we sieve with it in Section \ref{sec-th1} by extensive use of computers, completing the proof.
The proof of Theorem \ref{th2} is basically similar to that of Theorem \ref{th1} and it is established in Section \ref{sec-th2}.
Sections 7 and 8 are devoted to prove Theorem \ref{th3}, where some other kinds of Baker's method and a striking result of Scott \cite{Sc} restricting parity information on unknown exponents appearing in the left-hand side of equations are also used as well as some results on ternary Diophantine equations based on the so-called modular approach.
In the final section we make some remarks on our results with a few problems for readers.

All computations in this paper were performed by a computer\footnote{Intel Core 7 11800H processor (with 8 cores) and 16GB of RAM} using the computer package MAGMA \cite{BoCaPl}.
The total computation time was about 2 weeks.

\section{Reducing to weak forms} \label{sec-reduce}%

Let $M$ and $\mathcal A$ be as in Definition $\ref{extend}$.
The extended multiplicative order of $\mathcal A$ modulo $M$, denoted by $e_M(\mathcal A)$, has similar properties to those of the multiplicative order of $\mathcal A$ modulo $M$.
We state some of them in the following lemma without their proof.

\begin{lem} \label{property}
Assume that $M>2.$
Define $\epsilon_0 \in \{1,-1\}$ by $\mathcal A^{e_M(\mathcal A)} \equiv \epsilon_0 \pmod{M}.$
Then the following hold.
\begin{itemize}
\item[\rm (i)]
Assume that $\mathcal A^n \equiv \epsilon \pmod{M}$ for some $n \in \N$ and $\epsilon \in \{1,-1\}.$
Then $n$ is a multiple of $e_M(\mathcal A)$ and $\epsilon={\epsilon_0}^{n/e_M(\mathcal A)}.$
\item[\rm (ii)] If $\epsilon_0=-1,$ then the multiplicative order of $\mathcal A$ modulo $M$ equals $2\,e_M(\mathcal A).$
\item[\rm (iii)] $e_M(\mathcal A^k)=e_M(\mathcal A)/\gcd(e_M(\mathcal A),k)$ for any positive integer $k.$
\end{itemize}
\end{lem}

The next lemma gives a simple but non-trivial divisibility property of solutions.

\begin{lem} \label{trick}
Let $(x,y,z)$ be a solution to equation $\eqref{abc}.$
Then
\[ e_d(b)\,x \equiv 0 \mod{e_d(a)}, \quad e_d(a)\,y \equiv 0 \mod{e_d(b)} \]
for any positive divisor $d$ of $c$ with $d>2.$
\end{lem}

\begin{proof}
Put $e_a:=e_d(a)$ and $e_b:=e_d(b)$.
Use the congruence $a^x \equiv - b^y \pmod{d}$ obtained from equation \eqref{abc} reduced modulo $d$.
Raising both sides of this congruence to the $e_a$-th power, one finds that $b^{e_a y} \equiv \pm (a^{e_a})^x \equiv \pm 1 \pmod{d}$.
Lemma \ref{property}\,(i) tells one that $e_b \mid e_a y$.
One also obtains $e_a \mid e_b x$ by symmetry of $a,b$.
\end{proof}

The following lemma tells us that proving Conjecture $\ref{atmost1conj}$ is reduced to studying a specialization of it.
Indeed, this fact will be applied for each of the proofs of our three theorems in the forthcoming sections.

\begin{lem} \label{weakform}
Let $d$ be a positive divisor of $c$ with $d>2.$
Define the positive integers $A$ and $B$ by
\[
A=a^{\,e_d(a)/g}, \quad B=b^{\,e_d(b)/g},
\]
where $g=\gcd(e_d(a),e_d(b)).$
Then the following hold.
\begin{itemize}
\item[\rm (i)]
The number of solutions to equation $\eqref{abc}$ equals that to equation \eqref{abc} corresponding to the triple $(A,B,c).$
\item[\rm (ii)]
The extended multiplicative orders of $A$ and $B$ modulo $d$ equal $g.$
\item[\rm (iii)]
$(a,b,c)$ belongs to set \eqref{excep-set} if and only if $(A,B,c)$ belongs to the same set.
\end{itemize}
\end{lem}

\begin{proof}
Note that $A,B>1$ and $\gcd(A,B,c)=1$.
Put $e_a:=e_d(a),\,e_b:=e_d(b)$ for simplicity.\par
(i) Let $(x,y,z)$ be a solution to equation \eqref{abc}.
By Lemma \ref{trick}, it follows that
\[
x \equiv 0 \mod{e_a/g}, \quad y \equiv 0 \mod{e_b/g}.
\]
This leads to
\[
A^{ \,x / (e_a/g) } + B^{ \,y / (e_b/g) } = c^z.
\]
It is clear that $(\,x / (e_a/g),\,y / (e_b/g),\,z\,)$ is a solution to equation \eqref{abc} for the triple $(A,B,c)$.
This correspondence proves the assertion.\par
(ii) This follows from Lemma \ref{property}\,(iii).\par
(iii) This is immediate.
\end{proof}

Note that assertion (iii) of the above lemma can be applied for the subsets of \eqref{excep-set} appearing in the statements of Theorems \ref{th1} and \ref{th3}.

\section{Preliminaries for Theorem \ref{th1}} \label{sec-th1-pre}%

For the proof of Theorem \ref{th1}, it suffices to prove that $(a,b,c)$ has to equal one of $(5,3,2),(13,3,2),(5,2,3)$ and $(7,2,3)$, whenever equation \eqref{abc}, with each of $a$ and $b$ congruent to $1$ or $-1$ modulo $c$, has at least two solutions.
Indeed, suppose that this is established, and that equation \eqref{abc} has at least two solutions for some triple $(a,b,c)$ satisfying $\gcd(e_c(a),e_c(b))=1$.
Then Lemma \ref{weakform}\,(i,\,ii) with $d=c$ tells us that the equation $A^x+B^y=c^z$ has at least two solutions, with $e_c(A)=e_c(B)=\gcd(e_c(a),e_c(b))=1$.
Thus one may conclude that $(A,B,c)$ equals one of the four exceptional triples mentioned before, so that the same holds for $(a,b,c)$.
Therefore, we assume that each of $a$ and $b$ is congruent to $1$ or $-1$ modulo $c$.
Clearly it suffices to consider only the case where $c$ is not a perfect power, while noting that one can not always assume both $a$ and $b$ are not perfect powers.

For each $h \in \{a,b\}$, we can uniquely define $\delta_h \in \{1,-1\}$ by the following congruence:
\begin{eqnarray}\label{cong-c'}
h \equiv \delta_h \mod {c'},
\end{eqnarray}
where
\[ c':= \begin{cases}
\, 4 & \text{if $c=2$},\\
\, c & \text{if $c>2$}.
\end{cases}\]
Note that
\[ \max\{a,b\} \ge c'+1, \quad \min\{a,b\} \ge c'-1.\]
Below, we often let $h$ denote any of $a$ and $b$.

We begin with the following lemma.

\begin{lem}\label{coprime}
Let $(x,y,z)$ be a solution to equation $\eqref{abc}.$
Then the following hold.
\begin{itemize}
\item[\rm (i)]
One of the following cases holds.
\[ \left\{ \begin{array}{lllll}
\delta_a=1, &\delta_b=-1, &y \text{ is odd\,}; \\
\delta_a=-1, &\delta_b=1, &x \text{ is odd\,}; \\
\delta_a=-1, &\delta_b=-1, &x \not\equiv y \pmod{2}.
\end{array} \right. \]
\item[\rm (ii)]
$x$ and $y$ are relatively prime.
\end{itemize}
\end{lem}

\begin{proof}
(i) It is easy to see that $c^z \equiv 0 \pmod{c'}$ since $z>1$ if $c=2$.
By congruence \eqref{cong-c'} one reduces equation \eqref{abc} modulo $c'$ to see that ${\delta_a}^x \equiv -{\delta_b}^y \pmod{c'}$.
This congruence is actually an equality, that is, ${\delta_a}^x=-{\delta_b}^y$, since $\delta_a, \delta_b \in \{1,-1\}$ and $c' >2$.
This implies the assertion.\par
(ii) Suppose on the contrary that $x,y$ have some common prime factor, say, $p$.
Note that $p$ is odd by (i).
Write $x=p x_0,y=p y_0$.
Define positive integers $L,R$ as follows:
\[
L:=a^{x_0}+b^{y_0}, \quad R:=\frac{(a^{x_0})^p+(b^{y_0})^p}{a^{x_0}+b^{y_0}}.
\]
Note that $R$ is odd with $R>p$.
Equation \eqref{abc} becomes
\begin{equation}\label{factor}
L \cdot R=c^z.
\end{equation}
By elementary number theory we know that $\gcd(L,R) \mid \{1,p\}$ and that $p \parallel R$ if $\gcd(L,R)=p$ (cf.~Lemma \ref{padic-lem} with $(U,V,N)=(a^{x_0},-b^{y_0},p)$ below).

Let $r$ be any prime factor of $c$.
It is obvious that $r$ divides $L$ or $R$, in particular,
\[
a^{x_0} + b^{y_0} \equiv 0 \mod{r} \quad \text{or} \quad (a^{x_0})^p + (b^{y_0})^p \equiv 0 \mod{r}.
\]
Since each of $a,b$ is congruent to $\pm 1$ modulo $r$ by the premise, and $p$ is odd, it follows that $(a^{x_0})^p \equiv a^{x_0} \pmod{r}$ and $(b^{y_0})^p \equiv b^{y_0} \pmod{r}$, so that
\[
a^{x_0} + b^{y_0} \equiv (a^{x_0})^p + (b^{y_0})^p \mod{r}.
\]
These congruences show that $L$ is divisible by $r$.

Applying the above argument with $r$ any prime factor of $R$, one concludes that $R$ has no prime factor other than $p$.
It follows that $R=p$, which is however absurd as $R>p$.
\end{proof}

Suppose that equation \eqref{abc} has two solutions, say $(x,y,z)$ and $(X,Y,Z)$, with $(x,y,z) \ne (X,Y,Z)$.
Then
\begin{eqnarray}
&a^x+b^y=c^z, \label{1st}\\
&a^X+b^Y=c^Z. \label{2nd}
\end{eqnarray}
Without loss of generality, we may assume that $z \le Z$.

From \eqref{1st} and \eqref{2nd}, it holds trivially that
\begin{equation}\label{trivial-ineqs}
x<\frac{\log c}{\log a}\,z, \ \ y<\frac{\log c}{\log b}\,z, \ \ X<\frac{\log c}{\log a}\,Z, \ \ Y<\frac{\log c}{\log b}\,Z.
\end{equation}

In what follows, we put
\[
\Delta:=|x Y-X y|.
\]
Note that $\Delta$ is nonzero in general (cf.~\cite[Lemma 3.3]{HuLe}).

In what follows, we put
\[ C:=
\begin{cases}
\,c & \text{if $c=2$ or $c \not\equiv 2 \pmod{4}$},\\
\,c/2 & \text{if $c>2$ and $c \equiv 2 \pmod{4}$}.
\end{cases} \]
It will turn out that the size of $C$ relative to $c$ is crucially important in several places in the proof of Theorem \ref{th1}.
In particular, we ask
\begin{equation}\label{essential}
C>\sqrt{c}, \quad \lim_{c \to \infty} \frac{C}{\sqrt{c}} = \infty.
\end{equation}

\begin{lem}\label{basic-cong}
$h^{\Delta} \equiv {\delta_h}^{\Delta} \pmod{C^z}$ for each $h \in \{a,b\}.$
\end{lem}

\begin{proof}
Since $z \le Z$, one reduces equations \eqref{1st}, \eqref{2nd} modulo $c^z$ to see that
\[
a^x \equiv -b^y \mod c^z, \quad a^X \equiv -b^Y \mod c^z,
\]
respectively.
From these observe that
\[
a^{xY} \equiv (-b^y)^Y \equiv (-1)^Y (b^Y)^y \equiv (-1)^Y (-a^X)^y \equiv (-1)^{y+Y} a^{Xy} \mod{c^z}.
\]
Similarly, $b^{xY} \equiv (-1)^{x+X} b^{Xy} \pmod{c^z}$.
Since $a,b$ are prime to the modulus, these congruences imply
\[
h^{\Delta} \equiv \varepsilon \mod c^z
\]
for some $\varepsilon \in \{1,-1\}$.
It suffices to show that $\varepsilon={\delta_h}^{\Delta}$.
Recall from congruence \eqref{cong-c'} that
\[ h \equiv \delta_h \mod {c'}. \]
Thus Lemma \ref{property}\,(i) tells that $\varepsilon={\delta_h}^{\Delta}$ if $\gcd(c^z,c')>2$.
It is clear that $\gcd(c^z,c')=c$ if $c>2$.
For $c=2$, since $z>1$ in equation \eqref{1st}, it follows that $\gcd(c^z,c')=\gcd(2^z,4)=4$.
To sum up, the lemma is proved.
\end{proof}

\begin{lem}\label{basic-cong-2}
For each $h \in \{a,b\},$ $h$ is congruent to $\delta_h$ modulo every prime factor of $c.$
Further, $h \equiv \delta_h \pmod{4}$ if either $c=2$ or $c \equiv 0 \pmod{4}.$
\end{lem}

\begin{proof}
The two assertions easily follow by reducing congruence \eqref{cong-c'} modulo every prime factor of $c$ and modulo 4, respectively.
\end{proof}

For any integer $M>1$, we denote by $\nu_M(\mathcal A)$ the $M$-adic valuation of a nonzero integer $\mathcal A$, that is, the highest exponent $e$ such that $M^e$ divides $\mathcal A$.
Further, if $p/q$ is a nonzero rational number with $p$ and $q$ coprime integers, we set $\nu_M(p/q):=\nu_M(p)-\nu_M(q)$.

The next lemma is well-known 
and gives a precise information on the $p$-adic valuations of integers in a special form.

\begin{lem} \label{padic-lem}
Let $p$ be a prime number.
Let $U$ and $V$ be relatively prime nonzero integers.
Assume that
\[
\begin{cases}
\, U \equiv V \mod{p} & \text{if $p \ne 2$},\\
\, U \equiv V \mod{4} & \text{if $p=2$}.
\end{cases}
\]
Then, for any positive integer $N,$
\[
\nu_p(U^N-V^N)=\nu_p(U-V)+\nu_p(N).
\]
\end{lem}

\begin{lem}\label{div}
$C^z$ divides $\gcd(a-\delta_a,b-\delta_b) \cdot \Delta.$
\end{lem}

\begin{proof}
Let $p$ be a prime factor of $C$.
If $p=2$, then $C$ is even, so that either $c=2$ or $c \equiv 0 \pmod{4}$.
Observe from Lemma \ref{basic-cong-2} that $h \equiv \delta_h \pmod{p}$, and that $h \equiv \delta_h \pmod{4}$ if $p=2$.
Then one may apply Lemma \ref{padic-lem} with $(U,V)=(h,\delta_h)$ and $N=\Delta$, to see that
\[
\nu_p (h^{\Delta}-{\delta_h}^{\Delta})=\nu_p( h-\delta_h) + \nu_p(\Delta) =\nu_p\bigr( (h-\delta_h) \cdot \Delta \bigr).
\]
From Lemma \ref{basic-cong} it follows that
\[
\nu_p (C^z) \le \nu_p\bigr( (h-\delta_h) \cdot \Delta \bigr).
\]
This inequality holds for an arbitrary prime factor $p$ of $C$.
Therefore, $C^z$ divides $(h-\delta_h) \cdot \Delta$, and the assertion follows.
\end{proof}

When $C=c/2$, the above lemma lacks the 2-adic divisibility information.
The following lemma complements it (cf.~\cite[Lemma 4.2]{MiPi}).

\begin{lem}\label{complement}
Assume that $c$ is even.
Then ${c[2]}^z$ divides $\gcd(a-\delta_{a,4},b-\delta_{b,4}) \cdot \Delta,$ where $\delta_{a,4} \in \{1,-1\}$ and $\delta_{b,4} \in \{1,-1\}$ are defined as $a \equiv \delta_{a,4} \pmod{4}$ and $b \equiv \delta_{b,4} \pmod{4},$ respectively.
\end{lem}

\section{Upper bounds for solutions} \label{sec-th1-bound}%

For any algebraic number $\gamma$, we define the absolute logarithmic height of $\gamma$ as follows:
\[
{\rm h}(\gamma) =\frac{1}{[\mathbb{Q}(\gamma):\mathbb{Q}]}\,\Bigl(\, \log |c_0| \,+ \,\sum \,\log \max\bigr\{ 1,\,| \gamma' |\bigr\} \,\Bigl),
\]
where $c_0$ is the leading coefficient of the minimal polynomial of $\gamma$ over $\mathbb Z$, and the sum extends over all conjugates $\gamma'$ of $\gamma$ in the field of complex numbers.

We will make use of the following result which is a simple consequence of [Bu, Theorem 2;\,$\mu = 4$].

\begin{prop} \label{Bu-madic}
Let $M$ be a positive integer with $M>1.$
Let $\alpha_1$ and $\alpha_2$ be multiplicatively independent rational numbers such that $\nu_q(\alpha_1)=0$ and $\nu_q(\alpha_2)=0$ for any prime factor $q$ of $M.$
Assume that ${\rm g}$ is a positive integer with $\gcd({\rm g},M)=1$ satisfying
\[
\nu_q( {\alpha_1}^{{\rm g}}-1 ) \ge \nu_{q}(M), \quad \nu_q( {\alpha_2}^{{\rm g}}-1 ) \ge 1
\]
for any prime factor $q$ of $M.$
If $M$ is even, then further assume that
\[
\nu_2( {\alpha_1}^{{\rm g}}-1 ) \ge 2, \quad \nu_2( {\alpha_2}^{{\rm g}}-1 ) \ge 2.
\]
Let $H_1$ and $H_2$ be positive numbers such that
\[
H_j \ge \max \{ {\rm h}(\alpha_j),\log M \} \quad (j=1,2).
\]
Then, for any positive integers $b_1$ and $b_2$ with $\gcd(b_1,b_2,M)=1,$
\[
\nu_M( {\alpha_1}^{b_1} - {\alpha_2}^{b_2}) \le \frac{53.6\,{\rm g}\,H_1 H_2}{\log^4 M}\, \Bigr(\!\max \{ \log b'+\log \log M+0.64,\,4 \log M \} \Big)^2
\]
with $b'=b_1/H_2+b_2/H_1.$
\end{prop}

In the remaining of this section, let $(x,y,z,X,Y,Z)$ be a solution to the simultaneous system of \eqref{1st} and \eqref{2nd}.

\begin{lem} \label{Kc}
$Z<\dfrac{K_c (\log a) \log b}{\log^2 c}<\dfrac{K_c\,z^2}{x y},$ where $K_c$ is given by
\[ K_c= \begin{cases}
\,13100 & \text{if $c=2$},\\
\,7400 & \text{if $c=3$},\\
\,1900 & \text{if $c=5$},\\
\,12500 & \text{if $c=6$},\\
\,1100 & \text{if $c=7$},\\
\,3600 & \text{if $c=10$},\\
\,2000 & \text{if $c=14$},\\
\,\dfrac{857.6 \, \kappa_c\log^2 c}{\log^2 C} & \text{otherwise},
\end{cases} \]
where $\kappa_c=1$ if $c \equiv 2 \pmod{4},$ and $\kappa_c=\frac{\log c}{\log(c-1)}$ otherwise.
\end{lem}

\begin{proof}
First, observe from equation \eqref{2nd} that
\begin{equation}\label{Kc-lbound}
\nu_C (\varLambda) = Z,
\end{equation}
where $\varLambda:=a^X+b^Y$.
To obtain an upper bound for the left-hand side above, we apply Proposition \ref{Bu-madic} for $M=C$.
According to Lemma \ref{coprime}\,(i) for the solution $(X,Y,Z)$, we shall set the parameters $(\alpha_1,\alpha_2)$ and $(b_1,b_2)$ as follows:
\[ \left(\alpha_1,b_1,\alpha_2,b_2\right):= \begin{cases}
\,(a,X,-b,Y) & \text{if $\delta_a=1, \delta_b=-1$, $Y$\, is odd}, \\
\,(-a,X,b,Y) & \text{if $\delta_a=-1, \delta_b=1$, $X$\, is odd}, \\
\,(-a,X,-b,Y) & \text{if $\delta_a=\delta_b=-1$, $X \not\equiv Y \pmod{2}$}.
\end{cases} \]
Then ${\alpha_1}^{b_1} - {\alpha_2}^{b_2}=\pm \varLambda$, and both $\alpha_1,\alpha_2$ are congruent to 1 modulo $c'$ by congruence \eqref{cong-c'}.
Note that $C \mid c'$.
Since $c'$ is divisible by $4$ if $C$ is even, one may take ${\rm g}=1$.
Further, recall that
\[ \min\{a,b\} \ge c'-1 = \begin{cases}
\, 3 & \text{if $c=2$},\\
\, c-1 & \text{if $c>2$}.
\end{cases} \]
Since $\min\{a,b\}<C\,(=M)$ holds only when $C=c>2$ and $\min\{a,b\}=c-1$, one may set
\[ (H_1,H_2):= \begin{cases}
\,(\log a,\kappa_c\log{b})
& \text{if $a>b$},\\
\,(\kappa_c\log a,\log{b})
& \text{if $a<b$}.
\end{cases} \]
Since $\gcd(b_1,b_2)=\gcd(X,Y)=1$ by Lemma \ref{coprime}\,(ii), Proposition \ref{Bu-madic} gives
\begin{equation}\label{Kc-ubound}
\nu_{M}(\varLambda) \le \frac{53.6 \cdot 1 \cdot \kappa_c \log a\, \log b}{\log^4 C} \cdot \mathcal B^2,
\end{equation}
where
\[
\mathcal B=\max \biggl\{ \log \biggl( \frac{X}{H_2}+\frac{Y}{H_1} \biggr)+\log \log C+0.64, \,4\log C \biggl\}.
\]
Noting the latter two inequalities in \eqref{trivial-ineqs} and $\kappa_c \ge 1$, one has
\begin{align*}
&\log \biggl( \frac{X}{H_2}+\frac{Y}{H_1} \biggr)+\log \log C+0.64 \\
<&\log \biggl( \frac{Z(\log c)/\log a}{\log b}+\frac{Z(\log c)/\log b}{\log a} \biggr)+\log ({\rm e}^{0.64}\log C) \\
=&\log \biggl( \frac{2\,{\rm e}^{0.64}\log C}{\log c}\,T\biggl),
\end{align*}
where ${\rm e}=\exp(1)$, and
\[
T:=\frac{\log^2 c}{\log a\,\log b}\,Z.
\]
Thus \eqref{Kc-lbound}, \eqref{Kc-ubound} together lead to
\begin{equation} \label{ubound-T}
T < \frac{53.6 \, \kappa_c \log^2 c}{\log^4 C} \cdot {\mathcal B'}^2,
\end{equation}
where
\[
\mathcal B':=\log \,\max \biggr\{ \frac{2\,{\rm e}^{0.64}\log C}{\log c}\,T ,\,C^4 \biggr\}.
\]
It remains to find an absolute upper bound of $T$ for each $c$ by using \eqref{ubound-T}.

If $2\,{\rm e}^{0.64}(\log C)\,T \le C^4 \log c$, then \eqref{ubound-T} gives
\[
T < \frac{53.6 \, \kappa_c\log^2 c}{\log^4 C} \cdot (4\log C)^2 = \frac{857.6 \, \kappa_c\log^2 c}{\log^2 C},
\]
so that
\begin{equation}\label{Kc-1stcase}
T \le \min \biggl\{ \frac{C^4\log c}{2\,{\rm e}^{0.64} \log C}\,, \, \frac{857.6 \, \kappa_c\log^2 c}{\log^2 C} \biggl\}.
\end{equation}
While if $2\,{\rm e}^{0.64}(\log C)\,T>C^4 \log c$, then
\begin{equation} \label{Kc-2ndcase}
\frac{C^4 \log c}{2\,{\rm e}^{0.64}\log C} < T < \frac{53.6 \, \kappa_c \log^2 c}{\log^4 C} \cdot \log^2 \biggl(\frac{2\,{\rm e}^{0.64}\log C}{\log c}\,T \biggl).
\end{equation}
For each $c$, one can, by calculus, combine \eqref{Kc-1stcase}, \eqref{Kc-2ndcase} to find an upper bound for $T$ as asserted, where inequalities \eqref{Kc-2ndcase} are compatible only if $c \le 10$ or $c=14$.
\end{proof}

In what follows, we define $\Delta'$ and $\ell=\ell(c,z,\Delta')$ as follows:
\begin{align*}
&\Delta':=\gcd(\Delta,C^z), \\ &\ell:=\lcm (c',C^z/\Delta').
\end{align*}
By Lemma \ref{div}, together with congruence \eqref{cong-c'},
\begin{equation} \label{cong-ell}
\quad \quad h \equiv \delta_h \mod{\ell}
\end{equation}
for each $h \in \{a,b\}$.
In particular, $\min\{a,b\} \ge \ell-1$.

The following proposition corresponds to a special case of \cite[Theorem 1]{Bu}.
In the notation of \cite{Bu} it corresponds to the case where $x_1/y_1$ and $x_2/y_2$ are multiplicatively independent, $h=0$ and ${\rm g}=1$, where the numbers $\max\{|x_i|,|y_i|\}$ for $i=1,2$ appearing in \cite[(1)]{Bu} should be replaced by their logarithms, respectively.

\begin{prop} \label{Bu-madic-strong}
Let $M$ be a positive integer with $M>1.$
Let $\alpha_1$ and $\alpha_2$ be nonzero rational numbers which are multiplicatively independent.
Assume that
\begin{equation} \label{strong1}
\nu_q(\alpha_1-1) \ge \nu_q(M), \ \ \nu_q(\alpha_2-1) \ge 1
\end{equation}
for any prime factor $q$ of $M.$
If $M$ is even, then further assume that
\begin{equation} \label{strong2}
\nu_2(\alpha_1-1) \ge 2, \ \ \nu_2(\alpha_2-1) \ge 2.
\end{equation}
Put
\[
\varLambda={\alpha_1}^{b_1}-{\alpha_2}^{b_2},\]
where $b_1$ and $b_2$ are positive integers such that at least one of $b_1$ and $b_2$ is prime to $M.$
Let $K, L, R_1,R_2,S_1$ and $S_2$ be positive integers with $K \ge 3$ and $L \ge 2.$
Put $R=R_1+R_2-1$ and $S=S_1+S_2-1.$
Assume that
\begin{equation}\label{strong3}
R_1 S_1 \ge L,
\end{equation}
\begin{multline}\label{strong4}
\operatorname{Card}\, \{r b_2+s b_1\,|\, r \in \Z, s \in \Z, 0 \le r < R_2, 0 \le s < S_2 \}>(K-1)L.
\end{multline}
Then $\nu_M(\varLambda) \le KL-1,$ whenever
\begin{multline} \label{strong5}
K(L-1)\log M > (1+2w)\log (KL)+(K-1)\log{\beta}\\+\gamma L R\,{\rm h}(\alpha_1)+\gamma L S\,{\rm h}(\alpha_2),
\end{multline}
where $w$ is the number of distinct prime divisors of $M,$ and
\[
\beta=\frac{(R-1)b_2 + (S-1)b_1}{2}\left(\prod_{k=1}^{K-1}{k!}\right)^{-2/(K^2-K)}, \quad \gamma=\frac{1}{2}-\frac{KL}{6RS}.
\]
\end{prop}

Note that Lemma \ref{coprime}\,(ii) ensures that at least one of $X,Y$ is odd, and that at least one of $X,Y$ is prime to $C$ when $C$ is a prime power.
According to these, in the next lemma, we shall consider each of the following (not necessarily independent) three cases:
\[
{\rm (C1)} \ c=2, \ \ \ {\rm (C2)} \ c>2, \ \ \ {\rm (C3)} \ C^z/\Delta'>2.
\]

In what follows, we put \[ \mathcal X:=\max\{x,y\}.\]

The next lemma will be actually applied when $c$ is very small, and $z,\mathcal X,\Delta'$, an upper bound for $Z$ are given explicitly.

\begin{lem} \label{strong-applied}
Assume that $a>b$ and $C$ is a prime power.
Put
\[ M=\begin{cases}
\,4 & \text{for {\rm (C1)}},\\
\,C & \text{for {\rm (C2)}},\\
\,C^z/\Delta' & \text{for {\rm (C3)}}.
\end{cases} \]
Let $Z_u$ be an upper bound for $Z.$
Let $k>0$ be a real number and $L \ge 2$ be an integer.
Put $a_1,a_2,K$ and $B$ as follows\,$:$
\begin{gather*}
a_1=\frac{z \log c}{\log M}, \quad a_2=\frac{z \log c}{\mathcal X \log M}, \quad K=\lfloor{kLa_1a_2}\rfloor+1,\\
B=\log \log M+\log (Z_u/z) +\log \biggl( \frac{\mathcal X}{ \log(\ell+1)}+\frac{1}{\log(\ell-1)} \biggl)- \frac{1}{2}\log k+\varepsilon(K),
\end{gather*}
where $\varepsilon(K)=3/2+\log{\frac{(1+\sqrt{K-1})\sqrt{K}}{2K-2}}.$
Further, put $R_1,R_2,S_1,S_2,R,S,$ $\gamma,f_0, f_1, f_2, f_3$ and $f_4$ as follows\,$:$
\begin{align*}
&R_1=\lfloor \sqrt{La_2/a_1} \rfloor+1, \ R_2=\lfloor \sqrt{(K-1)La_2/a_1} \rfloor+1,\\
&S_1=\lfloor \sqrt{La_1/a_2} \rfloor+1, \ S_2=\lfloor \sqrt{(K-1)La_1/a_2} \rfloor+1,\\
&R=R_1+R_2-1, \ S=S_1+S_2-1, \ \gamma=\frac{1}{2}-\frac{KL}{6RS},\\
&f_0=K(L-1), \ f_1=\frac{3\log (KL)}{\log M}, \ f_2=\frac{(K-1)B}{\log M},\\
&f_3=\gamma L R a_1, \ f_4=\gamma L S a_2.
\end{align*}
If $K \ge 3$ and $f_0>f_1+f_2+f_3+f_4,$ then
\[ Z \le \begin{cases}
\,\max\bigr\{2KL-1, Z_2\bigr\} & \text{for {\rm (C1)}},\\
\,\max \bigr\{KL-1, Z_2 \bigr\}
& \text{for {\rm (C2)}},\\
\,\max \bigr\{KL (z-t)-1,Z_2 \bigr\}
& \text{for {\rm (C3)}},
\end{cases}\]
where $Z_2=\lfloor \sqrt{k}L z\, (a_1/\mathcal X+a_2)\rfloor+1$ and $t$ is the nonnegative integer defined as $C^t=\Delta'.$
\end{lem}

\begin{proof}
First, observe that
\begin{equation} \label{strong-applied-lbound}
\nu_M(\varLambda) \ge
\begin{cases}
\,\lfloor Z/2 \rfloor & \text{for {\rm (C1)},}\\
\,Z & \text{for {\rm (C2)},} \\
\,\lfloor Z/(z-t) \rfloor & \text{for {\rm (C3)},}
\end{cases}
\end{equation}
where $\varLambda:=a^X+b^Y$.
Indeed, since $\varLambda=c^Z$ by equation \eqref{2nd}, the above clearly holds for both cases (C1) and (C2), further, for (C3),
\[
\nu_M(\varLambda)=\nu_{C^{z-t}}(c^Z) \ge \nu_{c^{z-t}}(c^Z)=\lfloor Z/(z-t) \rfloor.
\]
To obtain upper bounds for the left-hand side of \eqref{strong-applied-lbound}, we will apply Proposition \ref{Bu-madic-strong}.
For this, set $(\alpha_1,\alpha_2)$ and $(b_1,b_2)$ in the same way as in the proof of Lemma \ref{Kc}.
Note that at least one of $b_1,b_2$ is prime to the prime power $M$ as $\gcd(b_1,b_2)=\gcd(X,Y)=1$ by Lemma \ref{coprime}\,(ii).
Recall that ${\alpha_1}^{b_1} - {\alpha_2}^{b_2}=\pm \varLambda$, and that $\alpha_1 \equiv \alpha_2 \equiv 1 \pmod{c'}$.
Next, we shall observe all conditions \eqref{strong1} to \eqref{strong5} required in Proposition \ref{Bu-madic-strong}.

In both cases (C1) and (C2), $M$ divides $c'$, so that $\alpha_1 \equiv \alpha_2 \equiv 1 \pmod{M}$, thereby condition \eqref{strong1} holds.
The same congruences hold also for case (C3) by congruence \eqref{cong-ell}.
These imply condition \eqref{strong2} since $M \not\equiv 2 \pmod{4}$ by assumption.
Condition \eqref{strong3} holds by the definitions of $R_1,S_1$.
To investigate the validity of condition \eqref{strong4}, we distinguish two cases.

\vspace{0.2cm}\noindent{\it Case I.}\,
${\rm Card} \{r Y+s X\,|\, 0 \le r < R_2, 0 \le s < S_2 \} < R_2 S_2$.\par
Clearly there exist two distinct pairs $(r_1,s_1)$ and $(r_2,s_2)$ of integers with $0 \le r_1,r_2 < R_2$ and $0 \le s_1,s_2 < S_2$ such that $r_1 Y+s_1 X=r_2 Y+s_2 X$.
Since $Y(r_1-r_2)=X(s_2-s_1)\,(\ne0)$ with $\gcd(X,Y)=1$, one has $X \mid r_1-r_2$ and $Y \mid s_2-s_1$, so that $X<R_2$ and $Y<S_2$.
Then
\[
X \le R_2-1=\lfloor \sqrt{(K-1)La_2/a_1} \rfloor \le \sqrt{k L a_1a_2 \cdot La_2/a_1}=\sqrt{k}La_2.
\]
Similarly, $Y \le \sqrt{k}La_1$.
Since, by equations \eqref{1st}, \eqref{2nd} with $b<a$,
\begin{align*}
Z&<\frac{1}{\log c}\, \log \bigr( 2 \max\{a^X,b^Y\} \bigr) \\
&<\frac{\log 2}{\log c}+\max \biggl\{ \frac{\log a}{\log c}\,X,\,\frac{\log b}{\log c}\,Y \biggl\}<1+\max \Bigl\{ z X,\,\frac{z}{\mathcal X}\,Y \Bigl\},
\end{align*}
one has
\begin{equation} \label{strong-applied-ubound1}
Z < \sqrt{k} L z (a_1/\mathcal X+a_2)+1.
\end{equation}

\vspace{0.2cm}\noindent{\it Case II.}\,
${\rm Card} \{r Y+s X \,|\, 0 \le r < R_2, 0 \le s < S_2 \} = R_2 S_2.$\par
Condition \eqref{strong4} holds by the definitions of $R_2,S_2$.
We shall check the last condition, namely, \eqref{strong5}, which is equivalent to
\[
f_0>f_1+f_2 \cdot \frac{\log \beta}{B}+f_3 \cdot \frac{\log a}{z\log c}+f_4 \cdot \frac{\mathcal X \log b}{z \log c},
\]
where $\beta$ is defined as in Proposition \ref{Bu-madic-strong}.
Since $\max\{a,b^{\mathcal X}\}<c^z$, the above inequality holds if $f_0>f_1+f_2+f_3+f_4$ and $\log \beta \le B$.
According to the proof of \cite[Lemme 13]{BuLa},
\[
\left(\prod_{k=1}^{K-1}{k!}\right)^{-2/(K^2-K)} \le \frac{{\rm e}^{3/2}}{K-1},
\]
whenever $K \ge 3$.
Also observe that
\begin{align*}
R-1 & = R_1-1 + R_2-1 \\
& \le \sqrt{L a_2 / a_1} + \sqrt{(K-1) L a_2 / a_1} \\
& = (1+\sqrt{K-1}) \sqrt{L a_1 a_2} \cdot 1/a_1
< (1+\sqrt{K-1}) \sqrt{K/k} \cdot 1/a_1.
\end{align*}
Similarly, $S-1 \le (1+\sqrt{K-1}) \sqrt{K/k} \cdot 1/a_2$.
These together with the inequalities $\min\{a,b\} \ge \ell-1,X<\frac{\log c}{\log a}\,Z_u, Y <\frac{\log c}{\log b}\,Z_u$ lead us to see that
\begin{align*}
\beta & = \frac{(R-1)Y + (S-1)X}{2} \cdot \left(\prod_{k=1}^{K-1}{k!}\right)^{-2/(K^2-K)}\\
&<\frac{(1+\sqrt{K-1}) \sqrt{K/k}}{2} \left( \frac{Y}{a_1}+\frac{X}{a_2} \right) \cdot \frac{{\rm e}^{3/2}}{K-1} \\
& = \exp( \varepsilon(K) ) \cdot 1/\sqrt{k} \cdot (Y+X \mathcal X) \cdot \frac{\log M}{z\log c}\\
& < \exp( \varepsilon(K) ) \cdot 1/\sqrt{k} \cdot \left( \frac{1}{\log b}+\frac{\mathcal X}{\log a} \right) \frac{Z_u\log M}{z} \le \exp(B),
\end{align*}
whenever $K \ge 3$.
To sum up, if $K \ge 3$ and $f_0>f_1+f_2+f_3+f_4$, then condition \eqref{strong5} holds, and Proposition \ref{Bu-madic-strong} gives
\begin{equation} \label{strong-applied-ubound2}
\nu_M(\Lambda) \le KL-1.
\end{equation}

Finally, the combination of \eqref{strong-applied-lbound}, \eqref{strong-applied-ubound1}, \eqref{strong-applied-ubound2} yields the asserted bounds for $Z$.
\end{proof}

Later we will be lead to distinguish two cases according to $\mathcal X>1$ or $\mathcal X=1$.
For the latter case, we can find another application of Proposition \ref{Bu-madic} as follows:

\begin{lem}\label{bound-x1y1}
If $x=1$ and $y=1,$ then $Z<858\,z.$
\end{lem}

\begin{proof}
First, from equation \eqref{1st} with $x=y=1$, observe that
\[
-a/b-1=-c^z/b, \ \ -b/a-1=-c^z/a.
\]
Since both $a,b$ are prime to $c$, the above equalities in particular show that rationals $-a/b,-b/a$ are very close to 1 in $c$-adic sense.
This is a key idea in the proof.

By Lemma \ref{coprime}\,(i), we are in one of the following cases:
\[ \left\{ \begin{array}{lllll}
\delta_a=-1, &\delta_b=1, &X \text{ is odd}; \\
\delta_a=1, &\delta_b=-1, &Y \text{ is odd}.
\end{array} \right. \]
In each of these cases, to obtain an upper bound for $Z$, we will apply Proposition \ref{Bu-madic} for $M:=c^z$ in a different way from that of the proof of Lemma \ref{Kc}.
Note that $X \ne Y$ and $M \not\equiv 2 \pmod{4}$.
We shall set the parameters $(\alpha_1,\alpha_2)$ and $(b_1,b_2)$ as follows:
\[ \left(\alpha_1,b_1,\alpha_2,b_2\right):= \begin{cases}
\,(-a/b,X,b^{\,{\rm sgn}(Y-X)},|X-Y|) & \text{if $\delta_a=-1$},\\
\,(-b/a,Y,a^{\,{\rm sgn}(X-Y)},|X-Y|) & \text{if $\delta_a=1$},
\end{cases}\]
where ${\rm sgn}$ is the sign function.
Put $\varLambda:={\alpha_1}^{b_1} - {\alpha_2}^{b_2}$.
Then $\varLambda \in \{\pm c^Z / b^X\,, \pm c^Z / a^Y\}$, so that
\begin{equation}\label{bound-x1y1-lbound}
\nu_M (\varLambda) = \nu_{c^z} (c^Z) = \left\lfloor \frac{Z}{z} \right\rfloor.
\end{equation}

On the other hand, as remarked in the beginning, one has $\nu_M (\alpha_1-1) \ge 1$.
Also, $\nu_{c'} (\alpha_2-1) \ge 1$ by congruence \eqref{cong-c'}.
Thus one may take ${\rm g}=1$.
Further, since $\max\{a,b\}<c^z=M$, one may set $H_1:=\log M$ and $H_2:=\log M$.
To sum up, by noting that $\gcd(b_1,b_2)=\gcd(X,Y)=1$ by Lemma \ref{coprime}\,(ii), Proposition \ref{Bu-madic} gives
\begin{equation}\label{bound-x1y1-ubound}
\nu_M (\varLambda) \le \frac{53.6}{z^2\log^2 c} \cdot \mathcal B^2,
\end{equation}
where
\[
\mathcal B=\log\,\max \bigr\{ {\rm e}^{0.64}(b_1+|X-Y|), \, c^{4z} \bigr\}.
\]
Since $\max\{b_1,|X-Y|\} \le \max\{X,Y\}<\frac{\log c}{\log \min\{a,b\}}\,Z$, one has
\[
\mathcal B \le \log \,\max \biggr\{ \frac{2\,{\rm e}^{0.64} \log c}{\log (c'-1)}\,Z, \, c^{4z} \biggl \}.
\]
Thus \eqref{bound-x1y1-lbound}, \eqref{bound-x1y1-ubound} together lead to
\begin{equation} \label{Z/z}
\left\lfloor \frac{Z}{z} \right\rfloor \le \frac{53.6}{z^2 \log^2 c} \cdot {\mathcal B'}^2,
\end{equation}
where
\[
\mathcal B':=\log\,\max \biggr\{ \frac{2\,{\rm e}^{0.64} \log c}{\log (c'-1)}\,Z, \, c^{4z} \biggl\}.
\]

If $2\,{\rm e}^{0.64}(\log c)Z \le c^{4z}\log (c'-1)$, then \eqref{Z/z} yields
\[
\left\lfloor \frac{Z}{z} \right\rfloor < \frac{53.6}{z^2 \log^2 c} \cdot (4z\log c)^2=857.6,
\]
which leads to the assertion.
While if $2\,{\rm e}^{0.64}(\log c)Z > c^{4z}\log (c'-1)$, then
\[
Z>\frac{c^{4z} \log (c'-1)}{2\,{\rm e}^{0.64} \log c}, \quad \left\lfloor \frac{Z}{z} \right\rfloor \le \frac{53.6}{z^2 \log^2 c} \cdot \log^2 \biggl( \frac{2\,{\rm e}^{0.64} \log c}{ \log (c'-1)}\,Z \biggl).
\]
It is not hard to see that there are only finitely many pairs $(c,z)$ satisfying the above inequalities (with $c^z \ge 2c'$).
Indeed, either $c=2$ and $z \le 3$, or $c=3$ and $z=2$.
This implies that $(a,b,c)$ is $(3,5,2),(4,5,3)$ or $(2,7,3)$, where the assertion holds by classical results in the literature (cf.~\cite{Na}).
\end{proof}

\section{Proof of Theorem \ref{th1}} \label{sec-th1}%

In this section, we solve the following system of the equations:
\begin{eqnarray}
&a^x+b^y=c^z, \label{1st-calX}\\ &a^X+b^Y=c^Z, \label{2nd-calX}
\end{eqnarray}
where $a,b,c$ are given positive integers such that each of $a,b$ is congruent to $1$ or $-1$ modulo $c$, and $x,y,z,X,Y,Z$ are unknown positive integers with $(x,y,z) \ne (X,Y,Z)$ and $z \le Z$.
It suffices to consider when $a>b$ and $c$ is not a perfect power.
We shall keep the used notation $\delta_a,\delta_b,c',\Delta,C,K_c,\Delta',\ell,\mathcal X$ and use the results established in the previous sections.
In particular, we will frequently and implicitly rely on inequalities \eqref{trivial-ineqs} and $\min\{a,b\} \ge \ell-1$, together with the following notation:
\[
\tau_b:=\frac{\log c}{\log b}, \quad \tau_\ell:=\frac{\log c}{\log(\ell-1)}, \quad \tau_c:=\frac{\log c}{\log(c'-1)} \ \biggl( \le \frac{\log 3}{\log 2} \biggl).
\]

We begin by giving some restrictions on the solutions to the system of equations \eqref{1st-calX} and \eqref{2nd-calX}.

\begin{lem} \label{Delta-ineqs}
Let $(x,y,z,X,Y,Z)$ be a solution to the simultaneous system of equations \eqref{1st-calX} and \eqref{2nd-calX}$.$
Then the following hold.
\begin{itemize}
\item[\rm (i)]
$\mathcal X<\tau z$ and $\max\{X,Y\}<\tau Z$ for $\tau \in \{\tau_b,\tau_\ell,\tau_c\}.$
\item[\rm (ii)]
$\Delta<K_c\,z.$
\item[\rm (iii)]
$\Delta<\tau \mathcal X Z$ for $\tau \in \{\tau_b,\tau_\ell,\tau_c\}.$
\end{itemize}
\end{lem}

\begin{proof}
(i) This follows from inequalities \eqref{trivial-ineqs}.\par
(ii) Observe that if $x Y>X y$ then $\Delta=|x Y - X y|<x \cdot Y$, so that
\begin{equation} \label{Delta-ele-upp}
\Delta<\frac{\log c}{\log a}\,z \cdot \frac{\log c}{\log b}\,Z=\frac{\log^2 c}{\log a\,\log b}\,z Z.
\end{equation}
This holds also for $x Y<X y$.
The assertion now easily follows from Lemma \ref{Kc}.\par
(iii) This holds from (i) since $\Delta<\mathcal X \max\{X,Y\}$.
\end{proof}

\begin{lem} \label{Cz-ineqs}
Let $(x,y,z,X,Y,Z)$ be a solution to the simultaneous system of equations \eqref{1st-calX} and \eqref{2nd-calX}$.$
Then
\[
C^z<K_c\,z\,\gcd(a-\delta_a,b-\delta_b)<K_c\,z\,(c^{z/\mathcal X}+1).
\]
\end{lem}

\begin{proof}
Since $C^z \mid G \cdot \Delta$ by Lemma \ref{div}, where $G:=\gcd(a-\delta_a,b-\delta_b)$, one has $C^z \le G \cdot \Delta$.
On the other hand,
\[
G \le \min\{a-\delta_a,b-\delta_b\} \le \min\{a,b\}+1<c^{z/\mathcal X}+1.
\]
This together with Lemma \ref{Delta-ineqs}\,(ii) yields the asserted inequalities.
\end{proof}

Note that the above lemma can give absolute upper bounds for both $c$ and $z$ only if $\mathcal X>1$, where the premise that the extended multiplicative orders of $a$ and $b$ modulo $c$ equal 1 is essentially used.
The same remark applies for the proofs of Lemmas \ref{sols-calX=1-Z<2z} and \ref{Zge2z} below.

The following lemma corresponds to applying Lemma \ref{complement} to complement the case where $C=c/2$.

\begin{lem}\label{2adic-iota}
Assume that $C=c/2.$
Let $(z,X,Y,Z)$ be a solution to the simultaneous system of equations \eqref{1st-calX} and \eqref{2nd-calX}$.$
Then
\[
h \equiv \delta_{h,4} \mod{2^{\iota}}
\]
for each $h \in \{a,b\},$ where $\delta_{h,4}$ is defined as in Lemma $\ref{complement}$ and
\[
\iota=\max \biggr\{2,\,z - \biggl \lfloor \frac{\log (\Delta / \Delta')}{\log 2} \biggl \rfloor \biggr\}.
\]
\end{lem}

\begin{proof}
Lemma \ref{complement} yields $\nu_2(h-\delta_{h,4}) \ge z-\nu_2(\Delta)$.
On the other hand, since $C$ is odd, one finds from the definition of $\Delta'$ that
\[
\nu_2(\Delta) = \nu_2(\Delta/\Delta') \le \frac{\log (\Delta / \Delta')}{\log 2}.
\]
The two inequalities together show the assertion.
\end{proof}

Below we distinguish two cases according to whether $\mathcal X>1$ or $\mathcal X=1$.

\subsection{Case where $x>1$ or $y>1$} \label{subsec-calX>1}

The aim of this subsection is to prove the following:

\begin{lem} \label{sols-calX>1}
All solutions to the simultaneous system of equations \eqref{1st-calX} and \eqref{2nd-calX} satisfying $\mathcal X>1$ are given by
\[ (x,y,z,X,Y,Z)=\begin{cases}
\,(1,3,5,3,1,7) & \text{for $(a,b,c)=(5,3,2)$},\\
\,(1,2,2,2,1,3) & \text{for $(a,b,c)=(5,2,3)$}.
\end{cases} \]
\end{lem}

By a technical reason for applying Lemma \ref{Kc} and Proposition \ref{Bu-madic-strong} and for making the presentation simple, we distinguish two cases according as $c$ takes some very small values or not.

\subsubsection{Case where $c \notin \{2,3,5,6,7,10,14\}$}
\label{subsec-calX>1-clarge}

We perform the algorithm consisting of the following three steps to sieve all possible cases of the system of equations \eqref{1st-calX} and \eqref{2nd-calX}.
Three other programs are based on the same algorithm.
The computation time was less than 20 minutes.

\noindent \vskip.2cm
\noindent {\bf Step 1.}
{\it Find all possible pairs $(z,\mathcal X)$ with a corresponding upper bound for $c.$}\par
Observe that $c \ge 11$ and $C/\sqrt{c} \ge 9/\sqrt{18}\,(>1)$.
Since $\mathcal X \ge 2$, the second inequality in Lemma \ref{Cz-ineqs} yields an absolute upper bound for $z$, namely, $z \le 13$.
Thus $\mathcal X$ is also finite as $\mathcal X<\tau_c\,z$ by Lemma \ref{Delta-ineqs}\,(i).
Moreover, since $\lim_{c \to \infty}C/\sqrt{c}=\infty$ by \eqref{essential}, an upper bound for $c$ corresponding to each of all possible pairs $(z,\mathcal X)$ can be found, say $c_1$.
Put the resulting triples $[z,\mathcal X,c_1]$ into a list, say $list1$, which contains the following 25 elements:
\begin{align*}
&[ 2, 2, 1000 ],
[ 3, 2, 1090 ],
[ 3, 3, 190 ],
[ 4, 2, 306 ],
[ 4, 3, 70 ],
[ 4, 4, 46 ],
[ 5, 2, 138 ],
[ 5, 3, 38 ],\\
&[ 5, 4, 26 ],
[ 5, 5, 22 ],
[ 6, 2, 82 ],
[ 6, 3, 26 ],
[ 6, 4, 18 ],
[ 6, 5, 14 ],
[ 6, 6, 14 ],
[ 7, 2, 54 ],
[ 7, 3, 18 ],\\
& [ 7, 4, 14 ],
[ 8, 2, 38 ],
[ 8, 3, 14 ],
[ 9, 2, 30 ],
[ 10, 2, 26 ],
[ 11, 2, 22 ],
[ 12, 2, 18 ],
[ 13, 2, 18 ].
\end{align*}

\noindent \vskip.2cm
\noindent {\bf Step 2.}
{\it Find all possible numbers $a,b,c,x,y,z,\Delta'.$}\par
First, for each element in $list1$, take any possible $c$ at most $c_1$.
Second, take any possible $\Delta'$ satisfying
\[
\Delta' < \Delta_u, \quad \Delta' \mid C^z,
\]
where $\Delta_u:=K_c\,z$.
The above inequality follows from Lemma \ref{Delta-ineqs}\,(ii) as $\Delta' \le \Delta$.
Third, for each possible $(\delta_a,\delta_b)$ restricted by Lemma \ref{coprime}\,(i), take any possible $b$ satisfying congruence \eqref{cong-ell} and $b \le b_1$, where $b_1:=\lfloor c^{z/\mathcal X}\rfloor$.
Here \eqref{cong-ell} is a key sieving relation.
Fourth, after checking a restriction on the size of $\mathcal X$ from Lemma \ref{Delta-ineqs}\,(i), take any possible $x$ and $y$ satisfying $\gcd(x,y)=1$ by Lemma \ref{coprime}\,(ii), and check whether $c^z-b^y$ is a $x$-th power, and put $a:=(c^z-b^y)^{1/x}$.
Finally, if $a$ and $\Delta'$ satisfy suitable conditions including the first inequality in Lemma \ref{Cz-ineqs}, put the tuple $[a,b,c,x,y,z,\Delta']$ into a new list.
Moreover in case where $C=c/2$, use the following congruence in Lemma \ref{2adic-iota}:
\begin{equation} \label{eff-sieve}
h \equiv \delta_{h,4} \mod{2^{\iota}}
\end{equation}
for each $h \in \{a,b\}$, where
\[
\iota=\max \biggr\{2,\,z - \biggl \lfloor \frac{\log (\Delta_1 / \Delta')}{\log 2} \biggl \rfloor \biggr\}
\]
with $\Delta_1$ any upper bound for $\Delta$.
The above congruence is efficient to sieve (only when $C=c/2$).
The resulting list, say $list2$, contains 3026 elements.
The program for this step in the case where $C=c/2$ is as follows:

\vspace{0.2cm}{\tt
begin

\hskip.2cm for each element in list1 do
\vskip.1cm
\hskip.2cm for $c:=11$ to $c_1$ do
\vskip.1cm
\hskip.2cm if $\text{IsPower}(c)=\text{false}$ and $c \ne 14$ then
\vskip.1cm
\hskip.2cm Create the set $H:=\{D_v: D_v < \Delta_u \text{ and } C^z \text{ mod } D_v =0\}$
\vskip.1cm
\hskip.2cm for $D_v$ in $H$ do
\vskip.1cm
\hskip.2cm $\ell:=\text{lcm}(c',C^z \text{ div } D_v)$;
\vskip.1cm
\hskip.2cm for $s$ in $[\,[1,-1],[-1,1],[-1,-1]\,]$ do
\vskip.1cm
\hskip.2cm $\delta_a:=s[1]$; \ $\delta_b:=s[2]$; \ $b_0:=\delta_b$;
\vskip.1cm
\hskip.2cm for $s_b$ in $[1,-1]$ do
\vskip.1cm
\hskip.2cm By the Chinese Remainder Theorem, calculate the least
\vskip.1cm
\hskip.2cm positive integer $b_0$ satisfying
\vskip.1cm
\hskip.2cm \hspace{2cm}$b_0 \equiv \delta_b \pmod{\ell}, \ b_0 \equiv s_b \pmod{2^\iota} $
\vskip.1cm
\hskip.2cm with $\Delta_1=\Delta_u$;
\vskip.1cm
\hskip.2cm for $b:=b_0$ to $b_1$ by $\ell$ do
\vskip.1cm
\hskip.2cm if $b >1$ and $\mathcal X<\tau_b\,z$ 
then
\vskip.1cm
\hskip.2cm for $x:=1$ to $\text{min}(\,\mathcal X,\,\text{floor}(\,z\,(\log c)/\log (b+1)\,)\,)$ do
\vskip.1cm
\hskip.2cm for $y:=1$ to $\mathcal X$ do
\vskip.1cm
\hskip.2cm if $\text{max}(x,y)=\mathcal X$ and $\text{gcd}(x,y)=1$ then
\vskip.1cm
\hskip.2cm if IsIntegral$(\,(c^z-b^y)^{1/x}\,)$ = true then
\vskip.1cm
\hskip.2cm $a:=(c^z-b^y)^{1/x}$;
\vskip.1cm
\hskip.2cm if $a>b$ and $(a-\delta_a)$ mod $\ell=0$
\vskip.1cm
\hskip.2cm and $C^z<\Delta_u\,\text{gcd}(a-\delta_a,b-\delta_b)$ then
\vskip.1cm
\hskip.2cm Sieve with \eqref{eff-sieve} with $h=a$ and $\Delta_1=\Delta_u$
\vskip.1cm
\hskip.2cm $\Delta':=D_v$ and put $[a,\delta_a,b,\delta_b,c,x,y,z,\Delta']$ into $list2$

end}

\vskip.3cm
In the above program the for-loop on $s_b$ consisting of 5 lines is omitted when $C \ne c/2$.

\noindent \vskip.3cm
\noindent {\bf Step 3.}
{\it Find all possible numbers $a,b,c,x,y,z,X,Y,Z.$}\par
Use the following restrictions:
\begin{gather}
{\delta_a}^X=-{\delta_b}^Y, \ \gcd(X,Y)=1,\label{alg-calX>1-csmall-step3-1}\\
{\delta_a}^{X-1}(a-\delta_a)X+{\delta_b}^{Y-1}(b-\delta_b)Y \equiv 0 \mod{c^2}, \label{alg-calX>1-csmall-step3-2}\\
z \le \alpha+\nu_2(\Delta) \ \ \text{if $C=c/2$}, \label{alfa}
\end{gather}
where $\alpha:=\min\{ \nu_2(a^2-1), \nu_2(b^2-1)\}-1$.
These follow from Lemma \ref{coprime}, reducing equation \eqref{2nd-calX} modulo $c^2$ (cf.~proof of Lemma \ref{Zge2z} below) and Lemma \ref{complement}, respectively.

First, for each element in $list2$, take any possible $X$ by using the upper bound $Z_1$ for $Z$ from Lemma \ref{Kc}, where $Z_1:=\lfloor K_c(\log a)(\log b)/\log^2 c \rfloor$.
Second, take any possible product of $x$ and $Y$.
If the difference between its value and $X \cdot y$ satisfies two conditions, then define $Y$ suitably.
Third, sieve with \eqref{alg-calX>1-csmall-step3-1}, \eqref{alg-calX>1-csmall-step3-2}.
Fourth, define $\Delta$ suitably, and sieve with \eqref{alfa} and the definition of $\Delta'$.
Finally, check whether $a^X+b^Y$ is a power of $c$ and find $Z$.
The program for this step is as follows:

\vspace{0.2cm}{\tt
begin

\hskip.2cm for each element in list2 do
\vskip.1cm
\hskip.2cm $Xu:=\text{floor}(Z_1 (\log c)/\log a)$; \ $Yu:=\text{floor}(\tau_b Z_1)$;
\vskip.1cm
\hskip.2cm for $X:=1$ to $Xu$ do
\vskip.1cm
\hskip.2cm for $x Y:=x$ to $x \cdot Yu$ by $x$ do
\vskip.1cm
\hskip.2cm if $xY-X \cdot y \ne 0$ and $(xY-X \cdot y)$ mod $\Delta'=0$ then
\vskip.1cm
\hskip.2cm $Y:=xY$ div $x$;
\vskip.1cm
\hskip.2cm sieve with \eqref{alg-calX>1-csmall-step3-1} and \eqref{alg-calX>1-csmall-step3-2}
\vskip.1cm
\hskip.2cm $\Delta:=\text{abs}(xY - X \cdot y)$;
\vskip.1cm
\hskip.2cm sieve with \eqref{alfa}
\vskip.1cm
\hskip.2cm if $\text{gcd}(\Delta,C^z)=\Delta'$ then
\vskip.1cm
\hskip.2cm $W:=a^X+b^Y$; $i:=0$; repeat;
\vskip.1cm
\hskip.2cm if $W$ mod $c$ $=0$ then
\vskip.1cm
\hskip.2cm $W:=W$ div $c$; $i:=i+1$;
\vskip.1cm
\hskip.2cm until $W$ mod $c$ $\ne 0$;
\vskip.1cm
\hskip.2cm if $W=1$ then $Z:=i$; print $[a,b,c,x,y,z,X,Y,Z]$
\vskip.1cm
end}

\vskip.3cm
It turns out there is no output, and this completes the proof of Lemma \ref{sols-calX>1} for the values of $c$ under consideration.

\subsubsection{Case where $c \in \{2,3,5,6,7,10,14\}$}
\label{subsec-calX>1-csmall}

Note that $C$ is a prime, so that $\Delta'=C^t$ for some integer $t$ with $0 \le t \le z$, and $\ell=\lcm(c',C^{z-t})$.
We perform the algorithm below consisting of the following four steps.
The computation time was less than 5 hours.

\noindent \vskip.2cm
\noindent {\bf Step 1.}
{\it For each $c,$ find all possible pairs $(z,\mathcal X).$}\par
This step is basically the same as Step 1 of Section \ref{subsec-calX>1-clarge}.
Since $C/\sqrt{c}>1$ for each $c$, and $\mathcal X \ge 2$, the second inequality of Lemma \ref{Cz-ineqs} yields an absolute upper bound for $z$, so that $\mathcal X$ is also finite as $\mathcal X<\tau_c\,z$.
Put the resulting triples $[c,z,\mathcal X]$ into a list, say $list1$, which contains 526 elements.

\noindent \vskip.2cm
\noindent {\bf Step 1/a.}
{\it For each element in $list1,$ find all possible $t.$}\par
For each element in $list1$, take any $t$ with $0 \le t \le z$ satisfying
\[
{\rm lcm}(c',C^{z-t})-1 \le b_1, \quad C^t<K_c\,z,
\]
where $b_1$ is defined as in Step 2 of Section \ref{subsec-calX>1-clarge}.
The first inequality above holds since $b \ge \ell-1$, and the second one is by Lemma \ref{Delta-ineqs}\,(ii) as $C^t=\Delta'$.
Put the resulting quadruples $[c,z,\mathcal X,t]$ into a new list, say $list1/a$, which contains 1322 elements.

\noindent \vskip.2cm
\noindent {\bf Step 1/b.}
{\it For each element in $list1/a,$ find an upper bound for $Z.$}\par
Take any element in $list1/a$.
In this step, in order to find a sharper upper bound for $Z$ than $Z_1$, where $Z_1:=\lfloor K_c\,z^2/\mathcal X \rfloor$ by Lemma \ref{Kc}, apply Lemma \ref{strong-applied} with $Z_u:=Z_1$ as follows.
Put $M$ as
\[ M:=\begin{cases}
\,4 & \text{if $t \ge z-1$ and $c=2$},\\
\,C & \text{if $t \ge z-1$ and $c>2$},\\
\,C^{z-t} & \text{if $t<z-1$}.
\end{cases}\]
These choices correspond to cases (C1), (C2), (C3), respectively.
After defining $a_1=a_1(c,z,M),a_2=a_2(c,z,M,\mathcal X)$ as in Lemma \ref{strong-applied}, for suitable $k$ and $L$, use all other notation in that lemma.
Here find the pair $(k,L)$ in the following way.
First, for each pair $(k_0,L_0)$ of integers satisfying $1 \le k_0 \le 60$ and $2 \le L_0 \le 35$, put $k:=k_0/15$ and $L:=L_0$ and check whether inequalities $K \ge 3$ and $f_0>f_1+f_2+f_3+f_4$ hold.
Among all such suitable pairs, take one for which the upper bound for $Z$ obtained from Lemma \ref{strong-applied} becomes the least, and redefine $Z_u$ by the value found in this way.
Iterate this procedure three times, and let $Z_1$ be the resulting $Z_u$.
Finally sieve with the inequality $C^t<\tau_\ell \mathcal X Z_1$ by Lemma \ref{Delta-ineqs}\,(iii).
Put the resulting tuples $[c,z,\mathcal X,t,Z_1]$ into a new list, say $list1/b$, which contains 700 elements.

\vskip.2cm \noindent {\bf Step 2.}
{\it Find all possible numbers $a,b,c,x,y,z,X,Y,Z.$}\par
For each element in $list1/b$, perform the same algorithms as in Steps 2 \& 3 of Section \ref{subsec-calX>1-clarge}, with only one natural modification.
Namely, for generating $list2$ the definition of $D_v$ is simply replaced by $D_v:=C^t$.
The final output coincides with the solutions described in Lemma \ref{sols-calX>1}, and this completes the proof of Theorem \ref{th1} for the case where $\mathcal X>1$.

\subsection{Case where $x=1$ and $y=1$}

Here we examine the system of equations \eqref{1st-calX} and \eqref{2nd-calX} with $(x,y)=(1,1)$, that is,
\begin{eqnarray}
&a+b=c^z, \label{1st-calX=1}\\ &a^X+b^Y=c^Z. \label{2nd-calX=1}
\end{eqnarray}
The notation in Section \ref{subsec-calX>1} are also used below.
The aim of this and the next subsections is to prove the following:

\begin{lem} \label{sols-calX=1}
All solutions to the simultaneous system of equations \eqref{1st-calX=1} and \eqref{2nd-calX=1} are given by
\[ (z,X,Y,Z) = \begin{cases}
\,(3,1,3,5),(3,3,1,7) & \text{for $(a,b,c)=(5,3,2)$},\\
\,(4,1,5,8) & \text{for $(a,b,c)=(13,3,2)$}, \\
\,(2,2,5,4) & \text{for $(a,b,c)=(7,2,3)$}.
\end{cases} \]
\end{lem}

Firstly, we finish the case where $Z<2z$.

\begin{lem} \label{sols-calX=1-Z<2z}
The only solution to the simultaneous system of equations \eqref{1st-calX=1} and \eqref{2nd-calX=1} satisfying $Z<2z$ is given by $(z,X,Y,Z)=(3,1,3,5)$ for $(a,b,c)=(5,3,2).$
\end{lem}

\begin{proof}
Note that $z>1$ as it is clear that $z<Z$.
From equations \eqref{1st-calX=1} and \eqref{2nd-calX=1},
\begin{equation} \label{zZ}
a^X+b^Y=\frac{(a+b)^2}{c^{2z-Z}}.
\end{equation}
This yields that $a^X<a^X+b^Y<4a^2/c^{2z-Z}$, whence
\begin{equation} \label{zZ2}
a^{X-2}c^{2z-Z}<4
\end{equation}
with $2z>Z$.
In particular, $X \le 2$.

Suppose that $X=2$.
Then $c \le 3$ and $Z=2z-1$ by \eqref{zZ2}, thereby \eqref{zZ} becomes $(a^2+b^Y) c=(a+b)^2$.
Reducing this equation modulo $b$ implies that $c \equiv 1 \pmod{b}$, so that $(b,c)=(2,3)$.
However the equation used previously does not hold as $3(a^2+2^Y)>(a+2)^2$.
Thus $X=1$.

Eliminating $a$ from equations \eqref{1st-calX=1}, \eqref{2nd-calX=1} yields
\begin{equation}\label{bYcZz}
b^Y-b=c^Z-c^z.
\end{equation}
Note that $Y \ge 3$ as $b^{Y-1} \equiv 1 \pmod{c^z}$ with $Y>1$ and $c^z=a+b>b$.
Since $2z>Z$, and $b \ge C^z/\Delta'-1$ by congruence \eqref{cong-ell}, it follows that
\[
c^{2z-1} \ge c^Z>b^Y \ge \biggl( \frac{C^z}{\Delta'}-1 \biggr)^Y,
\]
so that $C^z < \Delta'(c^{(2z-1)/Y}+1)$.
Since $\Delta' \le \Delta=Y-1$, one has
\[
C^z <(Y-1)(c^{(2/Y)z-1/Y}+1).
\]

Suppose that $Y \ge 4$.
Recall that $C>\sqrt{c}$ by \eqref{essential}.
The above displayed inequality together with the inequalities $4 \le Y<\tau_c Z$ and $Z<2z$, implies that $(c,z,Y) \in \{(3,2,4),(6,3,4),(6,3,5)\}$.
For each of these triples equation \eqref{bYcZz} does not hold for any possible $b,Z$ with $z<Z<2z$.

Finally, we shall examine the case where $Y=3$, where \eqref{bYcZz} is
\begin{equation}\label{bYcZz-2}
b(b^2-1)=c^z(c^{Z-z}-1).
\end{equation}
Since $b^2=1+K c^z$ for some integer $K \ge 1$, one has $b K=c^{Z-z}-1$.
Noting that $z \ge Z-z$, one reduces this equation modulo $c^{Z-z}$ and squares the resulting one to see that $K^2 \equiv 1 \pmod{c^{Z-z}}$.

If $K=1$, then $b^2=1+c^z$ and $b=c^{Z-z}-1$.
It is not hard to see that these two equations together lead to $(b,c,z,Z)=(3,2,3,5)$, whence $a=c^z-b=5$.
Suppose that $K>1$.
Since $K^2 \equiv 1 \pmod{c^{Z-z}}$, it follows that $K>\sqrt{c^{Z-z}}$.
Therefore,
\[
b=\sqrt{1+K c^z}>\sqrt{K \cdot c^z}>c^{\frac{Z+z}{4}} \ge c^{\frac{3Z+1}{8}}.
\]
On the other hand, from \eqref{bYcZz-2},
\[
b^3=c^Z \cdot \frac{1-1/c^{Z-z}}{1-1/b^2} <c^Z \cdot \frac{4}{3}.
\]
These inequalities together imply that $c^{Z+3}<(4/3)^8\,(<10)$, which clearly does not hold.
\end{proof}

\subsection{Case where $Z \ge 2z$} \label{subsec-Zge2z}

Here we examine the case where $Z \ge 2z$.
By Lemma \ref{coprime}\,(i), we can write $\delta=\delta_a=-\delta_b$ for some $\delta \in \{1,-1\}$.
In what follows, we put
\[ \mathcal D:=\frac{C^z}{\Delta'}. \]
From equation \eqref{1st-calX=1} and congruence \eqref{cong-ell}, the numbers $a$ and $b$ can be written as follows:
\begin{equation} \label{ab-form}
a=A \cdot \mathcal D+\delta, \quad b=B \cdot \mathcal D-\delta,
\end{equation}
where $A$ and $B$ are some positive integers satisfying
\[ A+B=\frac{c^z}{\mathcal D}.\]
One substitutes the forms of $a$ and $b$ in \eqref{ab-form} into equation \eqref{2nd-calX=1}:
\begin{equation} \label{ABcalDcZ}
(A \cdot \mathcal D+\delta)^X+(B \cdot \mathcal D-\delta)^Y=c^Z.
\end{equation}

\begin{lem}\label{Zge2z}
Let $(z,X,Y,Z)$ be a solution to the system of equations \eqref{1st-calX=1} and \eqref{2nd-calX=1} satisfying $Z \ge 2z.$
Then
\[ (C/\sqrt{c})^z \le {\Delta'} \sqrt{ \max\{X,Y\}}.\]
Further, $AX+BY \equiv 0 \pmod{\mathcal D}.$
\end{lem}

\begin{proof}
Since $c^Z$ is divisible by $\mathcal D^2$ as $Z \ge 2z$, one reduces equation \eqref{ABcalDcZ} modulo $\mathcal D^2$ to find that
\[ a_0+a_1\mathcal D \equiv 0 \mod{\mathcal D^2},\]
where $a_0=\delta^X+(-\delta)^Y$ and $a_1=\delta^{X-1} AX+(-\delta)^{Y-1} BY$.
Since $a_0=0$ by Lemma \ref{coprime}\,(i), one has $(-\delta)^{Y-1}=(-\delta)^Y/(-\delta)=-\delta^X/(-\delta)=\delta^{X-1}$, so that $a_1=\delta^{X-1} (AX+BY)$.
These together show that $a_1$ is divisible by $\mathcal D$, in particular,
\[ AX+BY \ge \mathcal D.\]
On the other hand, by \eqref{ab-form},
\[ AX+BY \le \max\{X,Y\}(A+B)=\max\{X,Y\} \cdot \frac{c^z}{\mathcal D}.\]
These bounds for $AX+BY$ together yield
\[ \frac{\mathcal D^2}{c^z} \le \max\{X,Y\}, \]
which is equivalent to the asserted inequality.
\end{proof}

Below, we distinguish two cases in the same way as in the case where $\mathcal X>1$.

\subsubsection{Case where $c \not \in \{2,3,5,6,7,10,14\}$} \label{subsec-calX=1-clarge}

We perform the algorithm consisting of the following three steps to sieve all possible cases of the system of equations \eqref{1st-calX=1} and \eqref{2nd-calX=1}.
Since it is similar to that of Section \ref{subsec-calX>1-clarge}, we thus mainly refer to distinct places in each step.
The computation time was about 11 days.

\noindent \vskip.2cm
\noindent {\bf Step 1.}
{\it Find all possible $z$ with a corresponding upper bound for $c.$}\par
Adopt $Z_1$ as the upper bound from Lemma \ref{bound-x1y1}, namely, $Z_1:=\lfloor 858\,z \rfloor$.
Use the inequality
\[ (C/\sqrt{c})^z <(\tau_c Z_1)^{3/2}. \]
This follows by Lemma \ref{Zge2z} since $\Delta' \le \Delta=|X-Y|<\max\{X,Y\} <\tau_c Z$.
It yields an absolute upper bound for $z$, namely, $z \le 19$.
Moreover, since $\lim_{c \to \infty}C/\sqrt{c}=\infty$ by \eqref{essential}, an upper bound for $c$ corresponding to each of $z$ can be found, say $c_1$.
Put the resulting pairs $[z,c_1]$ into a list, say $list1$, which contains 18 elements.

\noindent \vskip.2cm
\noindent {\bf Step 2.}
{\it Find all possible numbers $a,b,c,z,\Delta'.$}\par
Set $\Delta_u:=\lfloor \tau_c Z_1 \rfloor, b_1:=\lfloor {c^z}/2 \rfloor$ and $\Delta_l:=\lceil (C/\sqrt{c})^z/ \sqrt{\tau_c Z_1}\, \rceil$.
Further, to restrict the size of $D_v$, use the inequalities
\[
\Delta'<\tau Z_1, \quad (C/\sqrt{c})^z < {\Delta'} \sqrt{\tau Z_1}
\]
for $\tau \in \{\tau_b,\tau_c\}$.
The second inequality easily follows from Lemma \ref{Zge2z}.
Put the resulting tuples $[a,b,c,z,\Delta']$ into a new list, say $list2$, which contains about 31 million elements.
The program for this step for the case where $C=c/2$ is as follows:

\vspace{0.2cm}{\tt
begin

\hskip.2cm for each element in list1 do
\vskip.1cm
\hskip.2cm for $c:=11$ to $c_1$ do
\vskip.1cm
\hskip.2cm if $\text{IsPower}(c)=\text{false}$ and $c \ne 14$ then
\vskip.1cm
\hskip.2cm Create the set
\vskip.1cm
\hskip.2cm $H:=\{D_v: \Delta_l \le D_v < \Delta_u \text{ and } C^z \text{ mod } D_v=0\}$
\vskip.1cm
\hskip.2cm for $D_v$ in $H$ do
\vskip.1cm
\hskip.2cm $\ell:=\text{lcm}(c',C^z \text{ div } D_v)$;
\vskip.1cm
\hskip.2cm for $s$ in $[1,-1]$ do
\vskip.1cm
\hskip.2cm $\delta_a:=s$; \ $\delta_b:=-s$; \ $b_0:=\delta_b$;
\vskip.1cm
\hskip.2cm for $s_b$ in $[1,-1]$ do
\vskip.1cm
\hskip.2cm By the Chinese Remainder Theorem, calculate the least
\vskip.1cm
\hskip.2cm positive integer $b_0$ satisfying
\vskip.1cm
\hskip.2cm \hspace{2cm}$b_0 \equiv \delta_b \pmod{\ell}, \ b_0 \equiv s_b \pmod{2^\iota} $
\vskip.1cm
\hskip.2cm with $\Delta_1=\tau_c Z_1$;
\vskip.1cm
\hskip.2cm for $b:=b_0$ to $b_1$ by $\ell$ do
\vskip.1cm
\hskip.2cm if $b>1$ and $(C/\sqrt{c})^z/\sqrt{\tau_b Z_1}<D_v<\tau_b Z_1$ then
\vskip.1cm
\hskip.2cm Sieve with \eqref{eff-sieve} with $h=b$ and $\Delta_1=\tau_b Z_1$
\vskip.1cm
\hskip.2cm $a:=c^z-b$; \ $\Delta':=D_v$; \ put $[a,\delta_a,b,\delta_b,c,z,\Delta']$ into $list2$

end}

\vskip.3cm
In the above program the for-loop on $s_b$ is omitted when $C \ne c/2$.
One probably has no choice but to distinguish the cases to keep the number of elements in $list2$ under control due to its huge size.
The most time-consuming part was the case with $z=2$.

\noindent \vskip.3cm
\noindent {\bf Step 3.}
{\it Find all possible numbers $a,b,c,z,X,Y,Z.$}\par
This is basically the same as Step 3 of Section \ref{subsec-calX>1-clarge} applied to the elements in $list2$.
It turns out there is no output, and this completes the proof of Lemma \ref{sols-calX=1} for the values of $c$ under consideration.

\subsubsection{Case where $c \in \{2,3,5,6,7,10,14\}$} \label{subsec-calX=1-csmall}

We perform the algorithm consisting of the following several steps, where the notation in Section \ref{subsec-calX=1-clarge} are used.
It is basically similar to that of Section \ref{subsec-calX=1-clarge}.
The computation time was about 1 hour.

\noindent \vskip.2cm
\noindent {\bf Step 1.}
{\it For each $c,$ find all possible $z.$}\par
For each $c$, find all $z$ satisfying the inequality
\[
(C/\sqrt{c})^z<(\tau_c Z_1)^{3/2},
\]
where $Z_1$ is defined as in Step 1 of Section \ref{subsec-calX=1-clarge}.
Put the resulting pairs $[c,z]$ into a new list, say $list1$, which contains 235 elements.

\noindent \vskip.2cm
\noindent {\bf Step 1/a.}
{\it For each element in $list1,$ find all possible $t.$}\par
This is similar to Step 1/a of Section \ref{subsec-calX>1-csmall}.
Find all $t$ satisfying
\begin{equation}\label{ineq-calX=1-csmall-1/b}
\ell-1 \le b_1, \ \ C^t<\tau_{\ell} Z_1, \ \ (C/\sqrt{c})^z < C^t \sqrt{\tau_{\ell} Z_1},
\end{equation}
where $b_1$ is defined as in Step 2 of Section \ref{subsec-calX=1-clarge}.
Put the resulting triples $[c,z,t]$ into a new list, say $list1/a$, which contains 629 elements.

\noindent \vskip.2cm
\noindent {\bf Step 1/b.}
{\it For each element in $list1/a,$ find an upper bound for $Z.$}\par
Perform the same algorithm as in Step 1/b of Section \ref{subsec-calX>1-csmall} with $Z_u:=Z_1$, and sieve with the last two inequalities in \eqref{ineq-calX=1-csmall-1/b}.
Put the resulting quadruples $[c,z,t,Z_1]$ into a new list, say $list1/b$, which contains 351 elements.

\vskip.2cm \noindent {\bf Step 2.}
{\it Find all possible numbers $a,b,c,z,X,Y,Z.$}\par
For each element in $list1/b$, perform the same algorithms as in Steps 2 \& 3 of Section \ref{subsec-calX=1-clarge}, with the definition of $D_v$ replaced by $D_v:=C^t$.
The output, together with Lemma \ref{sols-calX=1-Z<2z}, are the solutions described in Lemma \ref{sols-calX=1}.

\vskip.3cm
This completes the proof of Theorem \ref{th1}.

\section{Proof of Theorem \ref{th2}} \label{sec-th2}%

The proof proceeds along similar lines to that of Theorem \ref{th1}.
Therefore, below we mainly discuss differences between them.
Put $P:=\prod_{p \in S}p$.
Then $P$ is a divisor of $c$, and $M_S=P$ or $4P$ according to cases (I) or (II).
Note that we are in the case where $c$ is even if $S$ is empty, and recall that
\begin{equation}\label{cS-bound}
c_S > \sqrt{c}.
\end{equation}

By Lemma \ref{weakform}\,(i,\,ii) for $d=P$, we may assume for case (I) that each of $a$ and $b$ is congruent to $1$ or $-1$ modulo $M_S$.
This is similar to case (II) except when $c$ is not divisible by 4, where $4P$ is not a divisor of $c$.
In this exceptional case, $c[S]=\frac{1}{2}c[S] \cdot 2 \ge \frac{1}{2}c[S \cup \{2\}]=c_S>\sqrt{c}$ by \eqref{cS-bound}, so that this case is reduced to case (I).

If $S$ is non-empty, then, since $P>2$, for each $h \in \{a,b\}$, we can uniquely define $\delta_h \in \{1,-1\}$ by the following congruence:
\begin{eqnarray}\label{ass-cong-th3}
h \equiv \delta_h \mod{P}.
\end{eqnarray}
When $c$ is even, since $h$ is odd, we can also uniquely define $\delta_{h,4} \in \{1,-1\}$ by the following congruence:
\begin{eqnarray}\label{ass-cong-mod4-th3}
h \equiv \delta_{h,4} \mod 4.
\end{eqnarray}
Note that $\delta_{h,4}=\delta_h$ in case (II).

We begin with the following lemma.

\begin{lem}\label{coprime-th3}
Let $(x,y,z)$ be a solution to equation $\eqref{abc}.$
Then the following hold.
\begin{itemize}
\item[\rm (i)]
$x$ or $y$ is odd.
\item[\rm (ii)]
$x,y$ and $P$ are relatively prime.
\end{itemize}
\end{lem}

\begin{proof}
(i) First consider the case where $S$ is non-empty.
By \eqref{ass-cong-th3} one reduces equation \eqref{abc} modulo $P$ to find that ${\delta_a}^x \equiv -{\delta_b}^y \pmod{P}$.
Thus ${\delta_a}^x=-{\delta_b}^y$, implying the assertion.
Finally, suppose that $S$ is empty and $c$ is even.
Similarly to the previous case, one can obtain the assertion by using \eqref{ass-cong-mod4-th3} and reducing equation \eqref{abc} modulo $4$, whenever $c^z \equiv 0 \pmod{4}$.
If $c^z \not\equiv 0 \pmod{4}$, then $2 \parallel c$ and $z=1$, so that $c=2$ since $\sqrt{c}<c_S \le c[2]=2$ by \eqref{cS-bound}, which clearly contradicts equation \eqref{abc}.\par
(ii) By (i) suppose that $x,y$ have some common odd prime factor belonging to $S$, say $p$.
We shall use the notation in the proof of Lemma \ref{coprime}\,(ii), and we take $r$ as any prime belonging to $S'$, where $S'=S \cup \{2\}$ if $c$ is even, and $S'=S$ if $c$ is odd.
Since each of $a,b$ is congruent to $\pm1$ modulo $r$, it turns out that $L$ is divisible by $r$.
Recalling that $R$ is odd, and that $\nu_r(R)>0$ if and only if $r=p$ with $\nu_p(R)=1$, one infers from equation \eqref{factor} that $(c[S'])^z/p$ divides $L$.
Replacing this divisibility relation by an inequality, one finds that
\[
\frac{R}{L} = \frac{c^z}{L^2} \le \frac{c^z}{\bigr( (c[S'])^z/p \bigr)^2}=p^2 \cdot \left( \frac{c}{c[S']^2}\right)^z.
\]
Since $c[S']=c[S \cup \{2\}] \ge c_S$, it follows from \eqref{cS-bound} that
\begin{equation}\label{ABp}
\frac{A^p+B^p}{(A+B)^2}< p^2,
\end{equation}
where $A:=a^{x_0}$ and $B:=b^{y_0}$.
Note that each of $A,B$ is congruent to $\pm 1$ modulo $p$.
If $A>B$ (the other case is similar), then
\[ A^p<A^p+B^p<p^2(A+B)^2<p^2(2A)^2, \]
so that $A^{p-2}<4p^2$.
Since $A \ge p+1$, one has $p=3$, and $A<36$.
However, none of all possible pairs $(A,B)$ satisfies \eqref{ABp}.
\end{proof}

Suppose that the system of equations \eqref{1st} and \eqref{2nd} has a solution $(x,y,z,X,Y,Z)$ with $(x,y,z) \ne (X,Y,Z)$ and $z \le Z$.
Define the positive integer $\Delta$ as in the proof of Theorem \ref{th1}.

\begin{lem}\label{div-th3}
${c[S]}^z$ divides $\gcd(a-\delta_a,b-\delta_b) \cdot \Delta,$ and ${c[2]}^z$ divides $\gcd(a-\delta_{a,4},b-\delta_{b,4}) \cdot \Delta.$
\end{lem}

\begin{proof}
In the same way as in the proof of Lemma \ref{basic-cong}, for each $h \in \{a,b\}$ one obtains the congruence $h^{\Delta} \equiv \varepsilon \pmod{c^z}$ for some $\varepsilon=\pm1$.
Reducing this congruence modulo each $p \in S$, and combining the resulting one with congruence \eqref{ass-cong-th3} yields $\varepsilon \equiv {\delta_h}^{\Delta} \pmod{p}$, so that $\varepsilon={\delta_h}^{\Delta}$, and $h^{\Delta} \equiv {\delta_h}^{\Delta} \pmod{p^z}$.
Now the first assertion follows similarly to the proof of Lemma \ref{div}.
The second one is just a redisplaying of Lemma \ref{complement}.
\end{proof}

In what follows, we frequently use the Vinogradov notation $f \ll g$, which means that $|f/g|$ is less than some positive absolute constant.
Note that for each Vinogradov notation appearing below the corresponding implied constant is effectively computable.

\begin{lem}\label{1st-bound-th2}
$Z \ll (\log \log c)^2(\log a)\log b.$
\end{lem}

\begin{proof}
We shall independently consider the two cases where $S$ is nonempty and $c$ is even.
According to these cases, we put $M:=P$ and $M:=4$, respectively.
From equation \eqref{2nd},
\[ \nu_{M}(a^{2X}-b^{2Y}) \gg Z. \]
To obtain an upper bound for the left-hand side above, we apply Proposition \ref{Bu-madic} for $(\alpha_1,\alpha_2):=(a^2,b^2),(b_1,b_2):=(X,Y)$.
Congruences \eqref{ass-cong-th3}, \eqref{ass-cong-mod4-th3} together with Lemma \ref{coprime-th3} enable one to take ${\rm g}:=1$.
Also, since $M \le \min\{a,b\}+1 \le \min\{a,b\}^2$, one may set $H_1:=2\log a$ and $H_2:=2\log b$.
Proposition \ref{Bu-madic} gives
\[
\nu_{M}(a^{2X}-b^{2Y}) \ll \frac{\log a\,\log b}{\log^4 M} \cdot \mathcal B^2,
\]
where
\[
\mathcal B=\max \biggl\{ \log \biggl( \frac{X}{2\log b}+\frac{Y}{2\log a} \biggr)+\log \log M, \,\log M \biggl\}.
\]
Since
\[
\log \biggl( \frac{X}{2\log b}+\frac{Y}{2\log a} \biggr)+\log \log M <\log \biggl( \frac{\log c \ \log M}{\log a \ \log b}\,Z\biggl),
\]
the bounds for $\nu_{M}(a^{2X}-b^{2Y})$ together lead to
\[ T \ll \frac{1}{\log^4 M} \cdot {\mathcal B'}^2, \]
where
\[
T:=\frac{Z}{\log a\,\log b}, \quad \mathcal B':=\log\,\max \bigr\{ (\log c)(\log M)\,T ,\,M \big\}.
\]
It is not hard to see from the above inequality that $T \ll (\log \log c)^2$.
\end{proof}

By Lemma \ref{1st-bound-th2} with inequality \eqref{Delta-ele-upp}, we estimate $\Delta$ from above as follows:
\[
\Delta<\frac{\log^2 c}{\log a\,\log b}\,zZ \ll (\log \log c)^2(\log c)^2 z.
\]
This shows that the size of $\Delta$ is much smaller than that of $c^z$.

Put $\Delta'$ and $\Delta_2$ as follows:
\[ \begin{cases}
\,\Delta':=\gcd(\Delta,{c[S]}^z), \ \ \Delta_2:=\gcd(\Delta,{c[2]}^z)
& \text{for case (I)},\\
\,\Delta':=\gcd(\Delta,(c[S]c[2])^z)
& \text{for case (II)}.
\end{cases}\]
By Lemma \ref{div-th3},
\[ \begin{cases}
\,h \equiv \delta_h \mod{{c[S]}^z/\Delta'}, \quad h \equiv \delta_{h,4} \mod{{c[2]}^z/\Delta_2}
& \text{for case (I)},\\
\,h \equiv \delta_h \mod{(c[S]c[2])^z/\Delta'}
& \text{for case (II)}
\end{cases}\]
for each $h \in \{a,b\}$.

In what follows, for case (I), we put
\[ \mathcal D:=\begin{cases}
\, {c[S]}^z/\Delta' & \text{if $c[S]>c[2]$},\\
\, {c[2]}^z/\Delta_2 & \text{if $c[2]>c[S]$},
\end{cases} \]
and $\mathcal D:=(c[S]c[2])^z/\Delta'$ for case (II).
Then, for each $h \in \{a,b\}$, it holds that
\begin{align}\label{cong-D-th3}
h \equiv \epsilon_h \mod{\mathcal D}
\end{align}
for some $\epsilon_h \in \{1,-1\}$.
Further,
\[
\mathcal D \ge \frac{{c_S}^z}{\Delta} \gg \frac{{c_S}^z}{(\log \log c)^2(\log c)^2 z}.
\]
In view of \eqref{cS-bound}, except for only finitely many pair $(c,z)$ being effectively determined, we may estimate $\mathcal D$ from below as follows:
\begin{equation}\label{calD-estimate}
\mathcal D>2, \quad \log \mathcal D \gg z \log c.
\end{equation}

In what follows, we assume that $\max\{a,b,c\}$ is sufficiently large so that at least one of $c$ and $z$ is sufficiently large, thereby inequalities \eqref{calD-estimate} hold.

\begin{lem}\label{second-bound-th2}
$z \cdot Z \ll \dfrac{(\log a)\log b}{\log^2 c}.$
\end{lem}

\begin{proof}
We proceed along similar lines to the proof of Lemma \ref{1st-bound-th2}, namely, we shall again apply Proposition \ref{Bu-madic}, with the same parameters other than $M$ as those in the proof of that lemma.
In this case, we set $M:=\mathcal D$.
This choice is justified by Lemma \ref{coprime-th3}, \eqref{cong-D-th3}, \eqref{calD-estimate}.
Proposition \ref{Bu-madic} gives
\[
\nu_{M}(a^{2X}-b^{2Y}) \ll \frac{\log a\, \log b}{\log^4 M} \cdot \mathcal B^2,
\]
where $\mathcal B$ is given as in the proof of Lemma \ref{1st-bound-th2}.
Since $\nu_{M}(a^{2X}-b^{2Y}) \gg Z/z$, the two inequalities together lead to
\[ T \ll \frac{z^2}{\log^4 M} \cdot {\mathcal B'}^2, \]
where
\[
T:=\frac{zZ}{\log a\,\log b}, \quad \mathcal B':=\log\,\max \bigr\{ (\log c)(\log M)\,T/z, \,M \bigr \}.
\]
If $(\log c)(\log M)\,T/z \le M$, then, since one may assume that $\log M=\log \mathcal D \gg z \log c$ by \eqref{calD-estimate},
\[
T \ll \frac{z^2}{\log^2 M} \ll \frac{1}{\log^2 c},
\]
proving the assertion.
While if $(\log c)(\log M)\,T/z>M$, by Lemma \ref{1st-bound-th2},
\[ \frac{M}{(\log c)\log M}<T/z \ll (\log \log c)^2, \]
implying that $\max\{c,z\} \ll 1$ as $\log M \gg z \log c$.
\end{proof}

Now, Lemma \ref{second-bound-th2} immediately yields
\[
z \cdot Z \ll \frac{\log a}{\log c} \cdot \frac{\log b}{\log c} \le \min \biggl\{\dfrac{z^2}{xy},\dfrac{Z^2}{XY} \biggl\},
\]
so that $Z \ll z$ and $\max\{x,y,X,Y\} \ll 1$.
In particular, $\Delta \ll 1$ and $\mathcal D \gg {c_S}^z$.

\begin{lem}\label{xy>1-th2}
If $\max\{x,y\}>1,$ then the second restriction for the exceptional triples stated in Theorem $\ref{th2}$ holds.
\end{lem}

\begin{proof}
Since $\mathcal D \le G$, where $G:=\gcd(a \pm 1,b \pm 1)$ for some possibles signs, and $\Delta \ll 1$, it follows that ${c_S}^z \ll G$.
On the other hand, $G \le \min\{a,b\}+1$.
Since $\min\{a,b\}<c^{\,z/\max\{x,y\}}$, the inequalities together lead to ${c_S}^z \ll c^{\,z/\max\{x,y\}}$.
Suppose that $\max\{x,y\}>1$.
Then ${c_S}^z \ll c^{z/2}$, so that $({c_S}/\sqrt{c})^z<\mathcal C$ for some absolute positive constant $\mathcal C$.
This together with \eqref{cS-bound} implies that
$c_S/\sqrt{c}<\mathcal C$.
Further, by equation \eqref{1st},
\[
\max\{a,b\}<c^z<c^{\,\frac{\log \mathcal C}{\log (\,c_S/\sqrt{c}\,)}} =\exp \biggl( \frac{2\log \mathcal C}{(\log c_S)/\log \sqrt{c}-1} \biggl).
\]
To sum up, the assertion holds by setting $\mathcal C_2$ as $2 \log \mathcal C$.
\end{proof}

By the above lemma, we may assume that $x=1$ and $y=1$.
Since $\mathcal D \mid c^z$ with $\mathcal D>2$, it turns out that $\epsilon_a=-\epsilon_b$ by reducing equation \eqref{1st} modulo $\mathcal D$ with congruence \eqref{cong-D-th3}.
Thus, $a$ and $b$ are written as in \eqref{ab-form} with $\delta$ replaced by $\epsilon_a$.

We finish the proof of Theorem \ref{th2} by proving the following lemma.

\begin{lem}
If $(x,y)=(1,1),$ then the same conclusion as that of Lemma $\ref{xy>1-th2}$ holds.
\end{lem}

\begin{proof}
It suffices to show that $({c_S}/\sqrt{c})^z \ll 1$ as seen in the proof of Lemma \ref{xy>1-th2}.

Suppose that $Z<2z$.
Assuming that $a>b$, one can closely follow the proof of Lemma \ref{sols-calX=1-Z<2z} to show that $X=1,Y \ge 4,$ further, $\mathcal D<c^{(2/Y)z-1/Y}+1 \le c^{z/2}$ by the fact that $b \ge \mathcal D -1$.
This leads to $D \ll c^{z/2}$, so that $({c_S}/\sqrt{c})^z \ll 1$.

Suppose that $Z \ge 2z$.
One can closely follow the proof of Lemma \ref{Zge2z} to show that $\mathcal D^2/c^z \le \max\{X,Y\}$.
This leads to $({c_S}/\sqrt{c})^z \ll 1$.
\end{proof}

\section{Preliminaries for Theorem \ref{th3}} \label{sec-th3-pre}%

Let $c \in \{5,17,257, 65537\}$.
Note that $c$ is an odd prime and $\varphi(c)=c-1$ is a power of 2.
Applying Lemma \ref{weakform} for $d=c$ as seen in the proof of Theorem \ref{th1}, we may assume that $e_c(a)=e_c(b)$.
Put $E:=e_c(a)=e_c(b)$.
Thanks to Theorem \ref{th1}, we may further assume that $E>1$.
Then
\begin{equation} \label{hE-cong}
h^E \equiv -1 \mod{c}
\end{equation}
for each $h \in \{a,b\}$.
Note that the multiplicative order of each of $a$ and $b$ modulo $c$ equals $2E$ (cf.~Lemma \ref{property}\,(ii)), in particular $E$ is a power of 2 as it divides $\varphi(c)/2$.

From now on, suppose that the following system of equations:
\begin{align}
a^x&+b^y=c^z,\label{1st-th3}\\ a^X&+b^Y=c^Z \label{2nd-th3}
\end{align}
holds for some positive integers $x,y,z,X,Y,Z$ with $(x,y,z) \ne (X,Y,Z)$.
Define $\Delta$ as in previous sections.
We show several lemmas below.

\begin{lem}\label{Delta-div}
$\Delta$ is divisible by $E.$
Further, $\Delta/E$ is odd if $x \not\equiv X \pmod{2}$ or $y \not\equiv Y \pmod{2}.$
\end{lem}

\begin{proof}
As observed in the proof of Lemma \ref{basic-cong},
\[
a^{\Delta} \equiv (-1)^{y+Y}, \ \ b^{\Delta} \equiv (-1)^{x+X} \mod{c}.
\]
By Lemma \ref{property}\,(i), this together with \eqref{hE-cong} implies the assertions.
\end{proof}

The next lemma directly follows from the above lemma, and it relies upon the fact that $\varphi(c)$ is a power of 2.

\begin{lem}\label{Deltaeven}
$\Delta$ is even.
\end{lem}

\begin{lem}\label{XYeven}
$x \not\equiv X \pmod{2}$ or $y \not\equiv Y \pmod{2}.$
Further, both $x$ and $y$ are even, or both $X$ and $Y$ are even.
\end{lem}

To prove this lemma, we rely on the following striking result of Scott which is a direct consequence of \cite[Lemma 6]{Sc}.

\begin{prop}\label{twoclass}
Assume that $c$ is a prime.
Let $(x_1,y_1,z_1)$ and $(x_2,y_2,z_2)$ be two solutions to equation $\eqref{abc}.$
Then $x_1 \not\equiv x_2 \pmod{2}$ or $y_1 \not\equiv y_2 \pmod{2},$ except when $(a,b,c)$ or $(b,a,c)$ equals one of $(5,3,2),(13,3,2)$ and $(10,3,13).$
\end{prop}

\begin{proof}[Proof of Lemma $\ref{XYeven}$]
Since $c \ne 2,13$, Proposition \ref{twoclass} is applied for the two solutions $(x,y,z)$ and $(X,Y,Z)$, and the first assertion clearly follows.
For the second one, consider the case where $x \not\equiv X \pmod 2$.
Since $xY \equiv Xy \pmod{2}$ by Lemma \ref{Deltaeven}, it follows that $y$ or $Y$ is even according as $x$ or $X$ is even.
The case where $y \not\equiv Y \pmod{2}$ is handled similarly.
\end{proof}

By the above lemma, without loss of generality, we may assume that both $X,Y$ are even, and that at least one of $x,y$ is odd.
Write \[ X=2X', \quad Y=2Y'.\]
Equation \eqref{2nd-th3} becomes
\begin{eqnarray} \label{2nd'-th3}
a^{2X'}+b^{2Y'}=c^Z.
\end{eqnarray}
Further, by Lemma \ref{Delta-div},
\begin{equation}\label{E/2divDel}
E \parallel \Delta.
\end{equation}

In what follows, we often argue over the ring of Gaussian integers.
We can write \[c=m^2+1,\] where $m=2^e$ with $e \in \{1,2,4,8\}$.
Also, put \[ \beta := m+i,\] where $i:=\sqrt{-1}$.
Note that $\beta$ is a prime element in $\mathbb Z[i]$, since $\beta$ divides the prime $c$.

In what follows, without loss of generality, we may assume that $a$ is odd and $b$ is even.

\begin{lem}\label{gauss-fac}
The following hold.
\begin{itemize}
\item[\rm (i)]
$\{a^{X'},b^{Y'}\}=\{\,|\operatorname{Re} (\beta^Z)|,|\operatorname{Im} (\beta^Z)|\,\}.$
More precisely,
\[
a^{X'}=\dfrac{1}{2}\,|\,\beta^Z+(-\bar{\beta})^Z\,|, \ \ b^{Y'}=\dfrac{1}{2}\,|\,\beta^Z-(-\bar{\beta})^Z\,|,
\]
where $\bar{\beta}$ denotes the complex conjugate of $\beta.$
\item[\rm (ii)]
$Y'=\dfrac{e+\nu_2(Z)}{\nu_2(b)}.$
\end{itemize}
\end{lem}

\begin{proof}
(i) Equation \eqref{2nd'-th3} is rewritten as follows:
\[ (a^{X'}+b^{Y'} i)(a^{X'}-b^{Y'} i) = c^Z. \]
A usual argument on the above factorization yields that $a^{X'}+b^{Y'} i= u \beta^Z$ for some unit $u \in \{\pm1,\pm i\}$.
This immediately shows the first assertion.
Further,
\[
a^{X'}=\frac{u}{2}\,(\beta^Z+\bar{\beta}^Z\cdot \bar{u}/u), \quad
\pm\,b^{Y'}=\frac{u}{2}\,(\beta^Z-\bar{\beta}^Z\cdot \bar{u}/u).
\]
Now the second assertion follows since $b$ is assumed to be even, and the difference between $\beta,-\bar{\beta}$ is divisible by 4. \par
(ii) Since numbers $\beta,-\bar{\beta}$ are coprime, and their difference is divisible by 4, one uses (i) to see that
\[
\nu_{2}(b^{Y'})=\nu_{2}\biggl( \frac{\beta+\bar{\beta}}{2}\biggr)+\nu_{2}\biggl( \frac{\beta^Z-(-\bar{\beta})^Z}{\beta-(-\bar{\beta})}\biggr)=\nu_{2}(m)+\nu_{2}(Z),
\]
leading to the assertion.
\end{proof}


\begin{lem}\label{Zsmall}
If $Z \le 3,$ then 
\[ (a,b)=(c-2,2), \ (x,y,z)=(1,1,1), \ (X,Y,Z)=(2,2e+2,2).\]
\end{lem}

\begin{proof}
We shall apply Lemma \ref{gauss-fac}\,(i) for each $Z \in \{1,2,3\}$.
It turns out that $Z>1$ as $\min\{a,b\}>1$, and that all possible pairs $(a^{X'},b^{Y'})$ satisfying $Z \le 3$ are given as follows:
\begin{align*}
(c,Z,a^{X'},b^{Y'}) \in \{ \, & (5,2,3,4),(17,2,15,8),(257,2,255,32),\\
&(65537,2,65535,512),(5,3,11,2),(17,3,47,52),\\
&(257,3,767,4048),(65537,3,196607,16776448)\,\}.
\end{align*}
If $Z=3$, then $(X',Y')=(1,1)$, so that equation \eqref{abc} corresponding to each of the above cases is one handled by \cite[Theorem 1]{CaDo}, which tells us that there is only one solution to it.
For the case where $Z=2$, one finds that $X'=1, a=c-2$ and $b$ is power of 2.
Equation \eqref{abc} corresponding to each of the cases is one handled by \cite[Theorem 1.4;\,$m=1$]{Miy}, and it turns out that $b=2$, and solutions $(x,y,z),(X,Y,Z)$ are given as asserted.
\end{proof}

By the above lemma, we may suppose in what follows that \[ Z \ge 4.\]
Below, we shall observe that this leads to a contradiction.

Although the following lemma seems to be dealt with by the existing methods for determining all square terms in (concrete) binary linear recurrent sequences (cf.~\cite{NaPe}), we choose to rely on results on ternary Diophantine equations based on the so-called modular approach.

\begin{lem}\label{X'Y'odd}
$X'$ and $Y'$ are odd.
\end{lem}

\begin{proof}
This follows from a simple application of the works \cite{BeElNg,Br,El} on the generalized Fermat equation (cf.~\cite[Ch.14]{Co}) of signature $(2,4,n)$ with $n \ge 4$ to equation \eqref{2nd'-th3}.
\end{proof}

\begin{lem}\label{X'1Y'1}
$X'=1$ or $Y'=1$ according as $Z$ is even or odd.
\end{lem}

\begin{proof}
The assertion for odd $Z$ holds by Lemmas \ref{gauss-fac}\,(ii) and \ref{X'Y'odd}.
Assume that $Z$ is even and observe the factorization $a^{2X'}=(c^{Z/2}+b^{Y'})(c^{Z/2}-b^{Y'})$ with $a$ odd.
Then
\[
c^{Z/2}+b^{Y'}=u^{2X'}, \ \ c^{Z/2}-b^{Y'}=v^{2X'}
\]
for some coprime odd positive integers $u,v$ with $u>v$.
Adding these equations leads to the following factorization involving rational integers:
\[ 2c^{Z/2}=(u^2+v^2) \cdot \frac{u^{2X'}+v^{2X'}}{u^2+v^2}, \]
where the fact that $X'$ is odd by Lemma \ref{X'Y'odd} is used.
By the primality of $c$, the above equation implies that the set of prime factors of $u^{2X'}+v^{2X'}$ (which is $\{2,c\}$) is included in that of $u^2+v^2$.
Then an old version of primitive divisor theorem of Zsigmondy (cf.~\cite{Zs}) is applied to the $X'$-th term of the sequence $\{(u^2)^t+(v^2)^t\}_{t \ge 1}$ to obtain $X'=1$.
\end{proof}

Note that the formula of Lemma \ref{gauss-fac}\,(i) helps us to easily find that $X'=1$ and $Y'=1$ for small values of $Z$ by checking that both $a^{X'}$ and $b^{Y'}$ are not perfect powers.

\begin{lem}\label{order8}
$E$ is divisible by $4.$
In particular, $\Delta$ is divisible by $4.$
\end{lem}

\begin{proof}
Suppose on the contrary that $4 \nmid E$, that is, $E=2$, so that $a^2 \equiv b^2 \equiv -1 \pmod{c}$.
However, in this case, it is observed from Lemma \ref{X'Y'odd} that the left-hand side of equation \eqref{2nd'-th3} should be congruent to $-2$ modulo $c$, which is clearly absurd.
\end{proof}

Note that the above lemma excludes the case where $c=5$.

\begin{lem}\label{xyodd}
$x$ and $y$ are odd.
\end{lem}

\begin{proof}
Since $\Delta=\pm 2(xY'-X'y)$, Lemmas \ref{X'Y'odd} and \ref{order8} together are used to find that $x-y$ is even.
This implies the assertion as $x$ or $y$ is already known to be odd.
\end{proof}

\begin{lem}\label{Evalue}
$E=E(e,Z)=2e/\gcd(2e,Z-1).$
\end{lem}

\begin{proof}
Recall that $c=m^2+1$, and $\{a^{X'},b^{Y'}\}=\{|\operatorname{Re} (\beta^Z)|,|\operatorname{Im} (\beta^Z)|\}$ by Lemma \ref{gauss-fac}\,(i).
Since $E$ is a power of $2$, and $X',Y'$ are odd, one finds from Lemma \ref{property}\,(iii) that
\[
e_c(a^{X'})=\frac{e_c(a)}{\gcd(e_c(a),X')}=\frac{E}{\gcd(E,X')}=E,
\]
and $e_c(b^{Y'})=E$ similarly.
Therefore, it suffices to show that
\begin{equation}\label{E-value}
\begin{cases}
\,\operatorname{Re} (\beta^Z) \equiv \pm 2^{Z-1} \mod{c} & \text{if $Z$ is even},\\
\,\operatorname{Im} (\beta^Z) \equiv \pm 2^{Z-1} \mod{c} & \text{if $Z$ is odd}.
\end{cases}
\end{equation}
Indeed, $e_c(\pm 2^{Z-1})=e_c(2^{Z-1})=e_c(2)/\gcd(e_c(2),Z-1)$, where $e_c(2)=2e$ as $2^{2e}=m^2 \equiv -1 \pmod{c}$.
For showing \eqref{E-value}, on the modulus $(m^2+1)$, observe the following.
If $Z$ is even or odd, then
\begin{align*}
\operatorname{Re} (\beta^Z)=\sum_{j=0}^{Z/2} \binom{Z}{2j}m^{Z-2j} i^{2j}
& \equiv \sum_{j=0}^{Z/2} \binom{Z}{2j}(m^2)^{Z/2-j}(-1)^j\\
& \equiv \sum_{j=0}^{Z/2} \binom{Z}{2j}(-1)^{Z/2} \equiv \pm 2^{Z-1},
\end{align*}
\begin{align*}
\operatorname{Im} (\beta^Z)&=\sum_{j=0}^{(Z-1)/2} \binom{Z}{2j+1}m^{Z-2j-1}(\sqrt{-1})^{2j}\\
& \equiv \sum_{j=0}^{(Z-1)/2} \binom{Z}{2j+1}(m^2)^{(Z-1)/2-j}(-1)^j\\
& \equiv \sum_{j=0}^{(Z-1)/2} \binom{Z}{2j+1}(-1)^{(Z-1)/2} \equiv \pm 2^{Z-1},
\end{align*}
respectively.
\end{proof}

Since $E \ge 4$ by Lemma \ref{order8}, it follows from Lemma \ref{Evalue} that $Z$ is even for $c=17$.
Further,
\begin{equation}\label{Ebounds}
4 \le E \le E_u,
\end{equation}
where $E_u=2e$ or $e$ according as $Z$ is even or odd.

\section{Proof of Theorem \ref{th3}} \label{sec-th3}%

We begin with the following lemma.

\begin{lem}\label{1stbound-th3}
The following hold.
\begin{itemize}
\item[\rm (i)]
$x,y$ and $c$ are relatively prime.
\item[\rm (ii)]
\[
z \le \max\biggl\{\frac{t_1 E^3}{\log^2 c},\,2.2\cdot10^4\biggr\}\, (\log a) \log b,
\]
where $t_1=53.6 \cdot 2 \cdot 4^2.$
\end{itemize}
These hold also for the solution $(X,Y,Z).$
\end{lem}

\begin{proof}
(i) Suppose on the contrary that $x,y$ are divisible by an odd prime $c$.
Then, it is observed, similarly to the proof of Lemma \ref{coprime}\,(ii), that $R=(a^x+b^y)/(a^{x/c}+b^{y/c})$ has to equal $c$.
This is absurd as $R>c$.\par
(ii) We proceed similarly to the proof of Lemma \ref{Kc}.
By equation $\eqref{1st-th3},$
\[ \nu_c(a^{2x}-b^{2y}) \ge z. \]
To obtain an upper bound for the left-hand side above, we shall apply Proposition \ref{Bu-madic} for $(\alpha_1,\alpha_2):=(a,b)$, $(b_1,b_2):=(2x,2y)$.
Note that $\gcd(b_1,b_2,c)=1$ by (i).
In this case, we set $M:=c$.
Since $E$ is a power of 2, one may set ${\rm g}:=2E$, and $H_1:=\log a', H_2:=\log b'$, where $a'=\max\{a,c\}$ and $b'=\max\{b,c\}$.
Then
\[
\nu_{c}(a^{2x}-b^{2y}) \le \frac{53.6\cdot 2E\,\log a'\,\log b'}{\log^4 c} \cdot \mathcal B^2,
\]
where
\[
\mathcal B=\max \biggl\{ \log \biggl( \frac{2x}{\log b'}+\frac{2y}{\log a'} \biggr)+\log \log c+0.64, \,4\log c \biggl\}.
\]
Observe that
\[
\log \biggl( \frac{2x}{\log b'}+\frac{2y}{\log a'} \biggr)+\log \log c+0.64 \le \log \biggl( \frac{4\,{\rm e}^{0.64} \log^2 c}{\log a\,\log b}\,z \biggl).
\]
The two bounds for $\nu_{c}(a^{2x}-b^{2y})$ together yield
\[
T \le 53.6 \cdot 2 \cdot \frac{\log a'}{\log a} \cdot \frac{\log b'}{\log b} \cdot \frac{E}{\log^4 c} \cdot {\mathcal B'}^2,
\]
where
\[
T:=\frac{z}{\log a\,\log b}, \quad \mathcal B':=\log\, \max \bigr\{ 4\,{\rm e}^{0.64} (\log^2 c)\,T ,\,c^4 \bigr\}.
\]
Since $a \ge (2c-1)^{1/E}$ and $b \ge (c-1)^{1/E}$ as $a^E \equiv b^E \equiv -1 \pmod{c}$ with $a$ odd, one easily observes that
\[
\frac{\log a'}{\log a} \le \frac{E\log c}{\log(2c-1)}, \quad \frac{\log b'}{\log b} \le \frac{E\log c}{\log(c-1)}.
\]
Therefore,
\[
T \le 53.6 \cdot 2 \cdot \frac{E^3}{\log^4 c} \cdot {\mathcal B'}^2,
\]
If $4\,{\rm e}^{0.64} (\log^2 c)\,T \le c^4$, then $\mathcal B' =4\log c$, so that
\begin{equation}\label{ineq-max<<}
T \le 53.6 \cdot 2 \cdot 4^2 \cdot \frac{E^3}{\log^2 c}
\end{equation}
Finally suppose that $4\,{\rm e}^{0.64} (\log^2 c)\,T>c^4$.
Then
\[
\frac{c^4}{4\,{\rm e}^{0.64} \log^2 c} < T \le \frac{53.6 \cdot 2 \cdot E^3}{\log^4 c} \cdot \log^2 \bigr(4\,{\rm e}^{0.64} (\log^2 c)\,T\bigr).
\]
Since $E \le 2e$ by Lemma \ref{Evalue}, the above inequalities together imply that $c=17$ and $T<2.2 \cdot 10^4$.
This together with \eqref{ineq-max<<} gives the assertion.
\end{proof}

In what follows, we put \[ \Delta':=\gcd(\Delta/E, c^{\,\min\{z,Z\}}).\]
Note that $\Delta'$ equals either 1 or a power of $c$.

\begin{lem}\label{DeltaDelta'}
The following hold.
\begin{itemize}
\item[\rm (i)]
$\Delta < \max\{ t_1 E^3, 2.2 \cdot 10^4\log^2 c\} \cdot \min\{z,Z\}.$
\item[\rm (ii)]
$\Delta' < \max\{t_1 E^2, \, 2.2 \cdot 10^4(\log^2 c)/E\} \cdot \min\{z,Z\}.$
\item[\rm (iii)]
$h^E \equiv -1 \pmod{ c^{\,\min\{z,Z\} } /\Delta' }$ for each $h \in \{a,b\}.$
\end{itemize}
\end{lem}

\begin{proof}
(i) This follows from Lemma \ref{1stbound-th3}\,(ii) with inequality \eqref{Delta-ele-upp}.\par
(ii) This follows from (i) since $\Delta' \le \Delta/E$.\par
(iii) Let $h \in \{a,b\}$.
We know that $h^{\Delta} \equiv \epsilon \pmod{c^{\,\min\{z,Z\}}}$ for some $\epsilon \in \{1,-1\}$.
Since $h^E \equiv -1 \pmod{c}$, it follows from Lemma \ref{property}\,(i) that $\epsilon=(-1)^{\Delta/E}$.
Lemma \ref{padic-lem} is applied for $p=c$ and $(U,V,N)=(h^E,-1,\Delta/E)$ to show that $c^{\,\min\{z,Z\}}$ divides $(h^E+1) \cdot \Delta/E$.
This yields the assertion.
\end{proof}

\begin{lem}\label{max<<min}
The following hold.
\begin{itemize}
\item[\rm (i)]
$Z<4z$ if $Z$ is even.
\item[\rm (ii)]
If $\Delta' \ge c^{\,\min\{z,Z\}/3},$ then
\[ \min\{z,Z\} \le \begin{cases}
\,10 & \text{for $c=17$}, \\
\,6 & \text{for $c=257$}, \\
\,3 & \text{for $c=65537$}.
\end{cases} \]
\item[\rm (iii)]
If $\Delta'<c^{\,\min\{z,Z\}/3},$ then
\[ \max\{z,Z\} < t_2 E \cdot \min\{z,Z\},\]
where $t_2=53.7 \cdot 2 \cdot 4^2 \cdot (3/2)^2.$
\end{itemize}
\end{lem}

\begin{proof}
(i) Assume that $Z$ is even.
Then $X'=1$ by Lemma \ref{X'1Y'1}.
Since $\{a,b^{Y'},c^{Z/2}\}$ forms a primitive Pythagorean triple, one has $c^{Z/2}<a^2$, so that $c^{Z/2}<c^{2z}$ by equation \eqref{1st}, whence $Z<4z$. \par
(ii) If $\Delta' \ge c^{\,\min\{z,Z\}/3}$, then, by inequalities \eqref{Ebounds} and Lemma \ref{DeltaDelta'}\,(ii),
\[
c^{\,\min\{z,Z\}/3}< \max \bigr\{ t_1 {E_u}^2, \, 2.2 \cdot 10^4(\log^2 c)/E_l \bigr\} \cdot \min\{z,Z\},
\]
where $E_l:=4, E_u:=2e$.
This implies the assertion. \par
(iii) We only consider the case where $z \le Z$ because the case where $z \ge Z$ is dealt with similarly by replacing $X,Y$ and $Z$ by $x,y$ and $z$, respectively.
The proof proceeds along similar lines to that of Lemma \ref{bound-x1y1}.
We shall apply Proposition \ref{Bu-madic} for $(\alpha_1,\alpha_2):=(a,b)$, $(b_1,b_2):=(2X,2Y)$ in a little rough manner.
In this case, we set $M:=c^z/\Delta'$.
From here we assume that $\Delta'<c^{\,z/3}$, that is,
\begin{equation}\label{M-low}
M>c^{\,2z/3}.
\end{equation}
By Lemma \ref{DeltaDelta'}\,(iii) one may take ${\rm g}:=2E$.
Since $\max\{a,b\}<c^z$, one may set $H_1:=z\log c$ and $H_2:=z\log c$.
Then
\[
\nu_{M}(a^{2X}-b^{2Y}) \le \frac{53.6 \cdot 2 E\log^2 c}{\log^4 M} \cdot z^2 \cdot \mathcal B^2,
\]
where
\[
\mathcal B=\max \biggl\{ \log \biggl( \frac{2X+2Y}{z\log c} \biggr)+\log (\log M)+0.64, \,4\log M \biggl \}.
\]
Observe that
\[
\mathcal B \le \log \max \{\kappa Z,\,M^4\}
\]
with $\kappa=\frac{4\,{\rm e}^{0.64}\log c}{\log \min\{a,b\}}$.
On the other hand,
\[
\nu_{M}(a^{2X}-b^{2Y}) \ge \left \lfloor \frac{Z}{z} \right \rfloor.
\]
Since we may assume that $Z/z$ is suitably large, the two bounds for $\nu_{M}(a^{2X}-b^{2Y})$ together lead to
\begin{equation}\label{ineq-max/min}
Z \le \frac{53.7 \cdot 2 E\log^2 c}{\log^4 M} \cdot z^3 \cdot {\mathcal B}^2.
\end{equation}

Suppose that $\kappa Z \le M^4$.
Then $\mathcal B=4\log M$.
Since $\log M>\frac{2}{3}z \log c$ by \eqref{M-low}, inequality \eqref{ineq-max/min} gives
\[
Z \le 53.7 \cdot 2 \cdot 4^2 \cdot \frac{E\log^2 c}{\log^2 M} \cdot z^3 < t_2 E\cdot z,
\]
showing the assertion.

Suppose that $\kappa Z>M^4$.
Since $E \le 2e$ by \eqref{Ebounds}, it follows from \eqref{ineq-max/min} that
\[
\frac{z Z}{\log^2 (\kappa Z)} \le \frac{53.7 \cdot 4e \cdot (3/2)^4}{\log^2 c}.
\]
However this is not compatible with the inequality $Z>M^4/\kappa\,(>c^{8z/3}/\kappa)$.
\end{proof}

For a number field $\mathbb K$ and a prime ideal $\pi$ in $\mathbb K$, we denote by $\nu_{\pi}(\alpha)$ the exponent of $\pi$ in the factorization of the fractional ideal generated by a nonzero element $\alpha$ in $\mathbb K$.


\begin{prop}[Th\'eor\`eme 3 of \cite{BuLa}] \label{BL}
Let $\mathbb K$ be a number field.
Let $\pi$ be a prime ideal in $\mathbb K,$ and $p$ the rational prime lying above $\pi.$
Let $\alpha_1$ and $\alpha_2$ be nonzero elements in $\mathbb K$ such that the fractional ideal generated by $\alpha_1 \alpha_2$ is not divisible by $\pi.$
Assume that $\alpha_1$ and $\alpha_2$ multiplicatively independent.
Let ${\rm g}$ be a positive integer such that
\[
{\alpha_1}^{\rm g} - 1 \in \pi, \quad {\alpha_2}^{\rm g} - 1 \in \pi.
\]
Let $H_1$ and $H_2$ be positive numbers such that
\[
\quad H_j \ge \max \biggl\{ \frac{D}{f_{\pi}}\,{\rm h}(\alpha_j), \, \log p \biggl\} \quad (j=1,2),
\]
where $D=[\mathbb Q(\alpha_1,\alpha_2):\mathbb Q]$ and $f_{\pi}$ is the inertia index of $\pi.$
Then, for any positive integers $b_1$ and $b_2,$
\begin{multline*}
\nu_{\pi}({\alpha_1}^{b_1}-{\alpha_2}^{b_2}) \le
\frac{24\,D^2 p\,{\rm g}\,H_1 H_2}{{f_{\pi}}^2(p-1)(\log p)^4}\\
\times \Bigr(\! \max\bigl\{ \log b'+\log \log p+0.4,{\textstyle \frac{10 f_{\pi}}{D}}\log p, 10\bigl\} \Big)^2
\end{multline*}
with $b'=b_1/H_2+b_2/H_1.$
\end{prop}

A bright idea of Luca \cite[Lemma 7]{Lu_aa_12} used in the proof of the following lemma together with the previous application of Proposition \ref{twoclass} plays the most important role to derive absolute upper bounds for the solutions.

\begin{lem}\label{min<<1}
Assume that $\Delta'<c^{\,\min\{z,Z\}/3}.$
Then the following hold.
\begin{itemize}
\item[\rm (i)]
If $z \le Z$ and $Z$ is odd, then
\[ z< \begin{cases}
\,1.2 \cdot 10^5 & \text{for $c=257$},\\
\,77000 & \text{for $c=65537$}.
\end{cases} \]
\item[\rm (ii)]
If $z \le Z$ and $Z$ is even, then
\[ Z< \begin{cases}
\, 6 \cdot 10^5 & \text{for $c=17$}, \\
\, 3 \cdot 10^5 & \text{for $c=257$},\\
\, 3.1 \cdot 10^5 & \text{for $c=65537$}.
\end{cases} \]
\item[\rm (iii)]
If $Z \le z,$ then
\[ Z< \begin{cases}
\,1.4 \cdot 10^5 & \text{for $c=17$}, \\
\,69000 & \text{for $c=257$},\\
\,77000 & \text{for $c=65537$}.
\end{cases} \]
\end{itemize}
\end{lem}

\begin{proof}
We know that $a^{\Delta} \equiv \pm1 \pmod{c^{\,\min\{z,Z\}}}$ with $\Delta$ even.
One raises this congruence to $2X'$-th power to find that
\[
(a^{4X'})^{\Delta/2} \equiv 1 \mod{c^{\,\min\{z,Z\}}}.
\]
Observe from Lemma \ref{gauss-fac}\,(i) that
\[
(a^{X'})^4=\dfrac{1}{2^4}\,|\,\beta^Z + (-\bar{\beta})^Z\,|^4=\dfrac{1}{2^4}\,(\,\beta^Z + (-\bar{\beta})^Z\,)^4,
\]
because the number $\beta^Z + (-\bar{\beta})^Z$ is either real or purely imaginary.
These together yield
\[
(\,\beta^Z+(-\bar{\beta})^Z\,)^{2\Delta} \equiv 2^{2\Delta} \mod{c^{\,\min\{z,Z\}}}.
\]
Recalling that $\bar{\beta}$ is a prime element dividing $c$, one reduces the above congruence modulo $\bar{\beta}^{\,\min\{z,Z\}}$ to obtain $
\beta^{2Z\Delta} \equiv 2^{2\Delta} \pmod{\bar{\beta}^{\,\min\{z,Z\}}}$, whence
\[
\nu_{\bar{\beta}} (\beta^{2Z\Delta}-2^{2\Delta}) \ge \min\{z,Z\}.
\]
To obtain an upper bound for the left-hand side above, we apply Proposition \ref{BL} for $\pi:=\bar{\beta}$, $(\alpha_1,\alpha_2):=(\beta,2)$ and $(b_1,b_2):=(2Z\Delta,2\Delta)$.
Note that $(p,f_{\pi},D)=(c,1,2)$.
Since $\beta \equiv 2i \pmod{\bar{\beta}}$, one may take ${\rm g}:=4e$.
Further, one may set $H_1:=\log c$ and $H_2:=\log c$ as ${\rm h}(\beta)=\frac{1}{2}\log c$.
Therefore,
\[
\nu_{\bar{\beta}} (\beta^{2Z\Delta}-2^{2\Delta}) \le \frac{t_3 \,c\,e}{(c-1)\log^2 c} \cdot \mathcal B^2,
\]
where $t_3:=24\cdot4\cdot2^2$, and
\[
\mathcal B=\log\,\max \bigr\{ 2\,{\rm e}^{0.4}\Delta(Z+1), \, c^5 \bigr\}.
\]
To sum up, the two bounds for $\nu_{\bar{\beta}} (\beta^{2Z\Delta}-2^{2\Delta})$ together yield
\begin{equation}\label{ineq-minzZ}
\min\{z,Z\} \le \frac{t_3\,c\,e}{(c-1)\log^2 c} \cdot \mathcal B^2.
\end{equation}
Below we mainly distinguish two cases.

\vspace{0.1cm}{\it Case where $z \le Z.$} \
If $2\,{\rm e}^{0.4}\Delta(Z+1) \le c^5$, then \eqref{ineq-minzZ} becomes
\begin{equation}\label{ineq-final-2-1}
z \le \frac{25 t_3\,c\,e}{c-1}.
\end{equation}
While if $2\,{\rm e}^{0.4}\Delta(Z+1)> c^5$, then
\begin{equation}\label{ineq-final-2-2}
Z+1>\frac{c^5}{2\,{\rm e}^{0.4}\Delta}, \quad z < \frac{t_3 \,c\,e\,\log^2 \bigr(2\,{\rm e}^{0.4}\Delta(Z+1)\big)}{(c-1)\log^2 c}.
\end{equation}
On the other hand, we know from Lemma \ref{DeltaDelta'}\,(i) and Lemma \ref{max<<min}\,(i,\,iii) that
\[ \Delta \le \Delta_u, \quad Z+1 \le T z, \]
respectively, where
\begin{gather*}
\Delta_u:=\max\{ t_1 {E_u}^3, 2.2 \cdot 10^4\log^2 c\} \cdot \min\{z,Z\},\\
T:=\begin{cases}
\, 4 & \text{for even $Z$}, \\
\, t_2 E_u & \text{for odd $Z$}.
\end{cases}
\end{gather*}
These together with \eqref{ineq-final-2-2} show that
\begin{equation}\label{ineq-final-2-3}
\frac{c^5}{2\,{\rm e}^{0.4}\,T\,\Delta_u}< z < \frac{t_3 \,c\,e\,\log^2 \bigr(2\,{\rm e}^{0.4}\,T\,\Delta_u\,z\big)}{(c-1)\log^2 c},
\end{equation}
The combination of \eqref{ineq-final-2-1}, \eqref{ineq-final-2-3} (with Lemma \ref{max<<min}\,(i)) implies assertions (i,\,ii).

\vspace{0.1cm}{\it Case where $Z \le z.$} \
Similarly to the previous case, inequality \eqref{ineq-minzZ} leads to either
\[ \begin{array}{lllll}
Z \le \min \biggl\{\dfrac{25 t_3\,c\,e}{c-1}, \dfrac{c^5}{8\,{\rm e}^{0.4}}-1 \biggl\}
\ \ \ \text{or}, \\
\dfrac{c^5}{2\,{\rm e}^{0.4}\Delta_u}-1<Z \le \dfrac{t_3\,c\,e\,\log^2 \bigr(2\,{\rm e}^{0.4}\Delta_u(Z+1)\big)}{(c-1)\log^2 c}.
\end{array} \]
This implies assertion (iii).
\end{proof}

The following is an easy consequence of \cite[Th\'eor\`eme 3]{LaMiNe} (cf.~\cite[Theorem 2.6]{Bu-book}).

\begin{prop} \label{bu-mig-thm2}
Let $\alpha$ be an algebraic number with $|\alpha|=1$ which is not a root of unity.
Put
\[
H(\alpha)=\max\bigr\{ D\,{\rm h}(\alpha)+22\,|\log \alpha|, \,40 \bigr\},
\]
where $D=[\mathbb{Q}(\alpha):\mathbb{Q}]$ and $\log$ denotes the principal value of the logarithm.
Then, for any positive integer $k,$
\[
\log|\alpha^k-1| \ge -\frac{9}{8}\,D^2\,H(\alpha)\,{\mathcal B}^2,
\]
where
\[
\mathcal B=\max\bigr\{ \log (k/25)+2.35+10.2/D, \, 34/D, \, 0.1/\sqrt{D/2}\, \bigr\}.
\]
\end{prop}

\begin{lem} \label{complex-baker}
Suppose that $Z>\chi z$ with a positive number $\chi>2$ and $Z$ is odd.
Then
\[
Z<\frac{9}{1-2/\chi} \left(1+\frac{22\pi}{\log c}\right) \bigl(\max \{ \log Z+4.3,\,17 \} \bigl)^2+1.
\]
\end{lem}

\begin{proof}
As seen in the proof of Lemma \ref{gauss-fac}\,(i), it holds that
\[
a^{X'}+b^{Y'} i=u \gamma^Z, \ \ a^{X'}-b^{Y'} i=\bar{u}\,{\bar{\gamma}}^{Z},
\]
where $u \in \{\pm1,\pm i\}$ and $\gamma \in \Z[i]$ is associated with $\beta$ or $\bar{\beta}$.
We may assume that $u=1$ since $Z$ is odd.
By eliminating the term $a^{X'}$ from the above two equations, since $Y'=1$ by Lemma \ref{X'1Y'1}, one has
\[
\biggl(\frac{\gamma}{\bar{\gamma}}\biggl)^Z-1=\frac{2bi}{{\bar{\gamma}}^{\,Z}}.
\]
Considering the absolute values of both sides above, since $b<c^z,|\bar{\gamma}|=|\beta|=c^{1/2}$, and $Z>\chi z$ by assumption, one obtains
\[
\left| (\gamma/\bar{\gamma})^Z-1 \right| < 2c^{z-Z/2}<2c^{-(1/2-1/\chi)Z}.
\]
To obtain a lower bound for the left-hand side above, we apply Proposition \ref{bu-mig-thm2} for $\alpha:=\gamma/\bar{\gamma}$ and $k:=Z$.
It is easy to see that the minimal polynomial of $\alpha$ over $\mathbb{Q}$ is $T^2 \pm(2-4/c)T+1$ for some sign.
From this it follows that $\alpha$ is quadratic and not an algebraic integer, further, ${\rm h}(\alpha)=\frac{1}{2}\log c$.
Since $|\log{\alpha}| \le \pi$, Proposition \ref{bu-mig-thm2} gives
\[
\log \left|(\gamma/\bar{\gamma})^Z-1 \right| \ge -\frac{\,9\,}{2}(\log c+22\pi)\,\bigl( \max \{ \log Z+4.3,\,17\}\bigl)^2.
\]
Finally, the two bounds for $|(\gamma/\bar{\gamma})^Z-1|$ together imply the assertion.
\end{proof}


\begin{lem} \label{Z-bound-sharp-thm2}
\[ Z< \begin{cases}
\, 6 \cdot 10^5 & \text{if $c=17$ and $Z$ is even}, \\
\, 3 \cdot 10^5 & \text{if $c=257$ and $Z$ is even},\\
\, 3.1 \cdot 10^5 & \text{if $c=65537$ and $Z$ is even},\\
\,2.8 \cdot 10^5 & \text{if $c=257$ and $Z$ is odd},\\
\,1.8 \cdot 10^5 & \text{if $c=65537$ and $Z$ is odd}.
\end{cases}\]
\end{lem}

\begin{proof}
The assertion for even $Z$ follows from the combination of Lemmas \ref{max<<min}\,(i,\,ii) and \ref{min<<1}\,(i,\,ii).
For the case where $Z$ is odd, we may assume by Lemmas \ref{max<<min}\,(ii) and \ref{min<<1}\,(i,\,iii) that $Z>\chi z$, where $\chi=2.29$ for $c=257$ and $\chi=2.24$ for $c=65537$.
Then applying Lemma \ref{complex-baker} yields the remaining assertions.
\end{proof}

Define the quantity $V$ as follows:
\[ V:=\begin{cases}
\,\nu_c (\,{a(\beta,Z)}^{2e}+1\,) & \text{if $Z$ is even},\\
\,\nu_c (\,{b(\beta,Z)}^{2e}-1\,) & \text{if $Z$ is odd},
\end{cases}\]
where
\[
a(\beta,Z):=\dfrac{1}{2}\,|\,\beta^Z+(-\bar{\beta})^Z\,|, \quad b(\beta,Z):=\dfrac{1}{2}\,|\,\beta^Z-(-\bar{\beta})^Z\,|.
\]
Number $V$ is an upper bound for $\min_{h \in \{a,b\}} \nu_c(h^E+1)$ by Lemmas \ref{gauss-fac}\,(i), \ref{X'1Y'1} and \ref{Evalue}, and it depends only on $c$ and $Z$ (recall that $c=m^2+1$ and $\beta=m+i$).
Many heuristic observations in the study of kinds of Wieferich primes predict that $V$ is expected to be very small in general. 
For each $c$ and each $Z$ bounded from above as in Lemma \ref{Z-bound-sharp-thm2}, we use a computer to calculate $V$ (within 9 hours), and the result reads as follows:

\begin{lem} \label{Wieferich}
\[ V \le \begin{cases}
\,5 & \text{for $c=17$}, \\
\,3 & \text{for $c=257$},\\
\,2 & \text{for $c=65537$}.
\end{cases} \]
\end{lem}


\begin{lem} \label{minzZ-very-sharp-thm2}
$Z<36000$ and
\[ \min\{z,Z\} \le \begin{cases}
\,10 & \text{for $c=17$}, \\
\,6 & \text{for $c=257$},\\
\,3 & \text{for $c=65537$}.
\end{cases} \]
\end{lem}

\begin{proof}
Lemma \ref{DeltaDelta'}\,(iii) says that $c^{\,\min\{z,Z\}}/\Delta'$ divides $h^E+1$ for each $h \in \{a,b\}$.
Since one may assume that $\Delta'<c^{\,\min\{z,Z\}/3}$ by Lemma \ref{max<<min}\,(ii), it follows that $
(2/3)\min\{z,Z\}<\min_{h \in \{a,b\}}\nu_c(h^E+1) \le V$.
Now the second assertion follows from Lemma \ref{Wieferich}.
From this, for the first assertion we may assume that $Z>200z$.
Then $Z$ is odd by Lemma \ref{max<<min}\,(i).
Applying Lemma \ref{complex-baker} with $\chi=200$ gives the first assertion.
\end{proof}

We are ready to complete the proof of Theorem \ref{th3}.

\begin{proof}[Proof of Theorem $\ref{th3}$]
First suppose that $z \le Z$.
Then $z \le 10$ by Lemma \ref{minzZ-very-sharp-thm2}.
Lemmas \ref{X'1Y'1} and \ref{Wieferich} say that $a=a(\beta,Z)$ or $b=b(\beta,Z)$, and that $Z \le 36000$, respectively.
Then one can use a computer to check that $\min\{a(\beta,Z),b(\beta,Z)\}>c^{10}$ whenever $Z \ge 21$.
From equation \eqref{1st-th3} it turns out that $Z<21$.
Now brute force computation suffices for checking that the system of equations \eqref{1st-th3}, \eqref{2nd'-th3} does not hold for any possible case (with $Z \ge 4$).

Finally suppose that $z>Z$.
We know that $Z$ is even with $Z \le 10$ for $c=17$, and that $Z \le 6$ for $c=257$.
Note that $c \ne 65537$ as $Z \ge 4$.
It is easy to see that $X'=1,Y'=1$, so that $a=a(\beta,Z),b=b(\beta,Z)$.
For each $m$ and each possible $Z$ one can fortunately find a positive odd integer $d>1$ satisfying either
\begin{equation}\label{Jacobi-2-th1}
\begin{split}
&d \mid a, \ \Big(\frac{b}{d}\Big)=-1, \ \left(\frac{c}{d}\right)=1 \quad \text{or}\\
&d \mid b, \ \left(\frac{a}{d}\right)=-1, \ \left(\frac{c}{d}\right)=1,
\end{split}
\end{equation}
where $\left(\frac{\cdot}{\cdot}\right)$ denotes the Jacobi symbol.
More precisely, $d$ (with $d \mid h$) is taken as in the following table:
\[
\begin{tabular}{c|ccccccc}
$c$ & 17 & 17 & 17 & 17 & 257 & 257 & 257 \\
$Z$ & 4 & 6 & 8 & 10 & 4 & 5 & 6 \\ \hline
$d$ & 15 & 15 & 15 & 19 & 15 & 139 & 11 \\
$h$ & $b$ & $a$ & $b$ & $b$ & $b$ & $a$ & $b$
\end{tabular}
\]
On the other hand, reducing equation \eqref{1st-th3} modulo such $d$ implies that
\begin{equation}\label{Jacobi-2-th2}
\begin{split}
\Big(\dfrac{b}{d}\Big)^x=\left(\dfrac{c}{d}\right)^z \ \ \text{if $d \mid a$},\\
\left(\dfrac{a}{d}\right)^x=\left(\dfrac{c}{d}\right)^z \ \ \text{if $d \mid b$}.
\end{split}
\end{equation}
However the combination of \eqref{Jacobi-2-th1}, \eqref{Jacobi-2-th2} implies that $x$ or $y$ is even, contradicting Lemma \ref{xyodd}.
\end{proof}

\section{Concluding remarks} \label{sec-rem}%

As mentioned in Section \ref{sec-intro}, each of Theorems \ref{th1} and \ref{th3} gives a 3-variable version of some work in \cite{Be_cjm_01}, and this complements the work of \cite{MiPi} in a sense.
Unfortunately the method of this paper seems not to be enough to consider similar versions for Bennett's other results \cite[Theorems 1.4 to 1.6]{Be_cjm_01}), for instance, for the case where $a$ or $b$ is fixed.
However, such a case seems to be much harder than the case where $c$ is fixed.
A reason for this is that even the case where $b=2$ in equation \eqref{pillai} is still difficult to be resolved due to its partial results obtained in \cite{Lu_indag_03,ScSt_jnt_04}.
Thus we surely need a new idea for this purpose.

Finally, for ambitious readers, we leave a few problems, concerning Theorems \ref{th1} and \ref{th3}, for which the method of this paper can be applied in principle.

\begin{prob}
Prove Conjecture $\ref{atmost1conj}$ for each of the following cases\,$:$
\begin{itemize}
\item[\rm (i)]
each of $a$ and $b$ is congruent to $1$ or $-1$ modulo $\prod_{p \mid c}p.$
\item[\rm (ii)]
$e_a(c)=e_b(c)$ with $e_a(c)$ even and $c$ is a prime.
\end{itemize}
\end{prob}

It seems that for handling (i) one needs a more extensive computation than that needed for proving Theorem \ref{th1}, and that for (ii) one needs to manage to find absolute upper bounds for corresponding solutions.

\subsection*{Acknowledgements}
We would like to thank the anonymous referee for the helpful remarks and suggestions.
We are also grateful to Mihai Cipu, Reese Scott, Robert Styer and Masaki Sudo for their many comments and remarks which improved an earlier draft.

\end{document}